\documentclass[preprint,10pt]{elsarticle}

\usepackage{amssymb}
\usepackage{amsthm}
\usepackage{graphics}
\usepackage{amsmath}
\usepackage[dvipsnames,usenames]{color}

\newtheorem{thm}{Theorem}[section]
\newtheorem{cor}[thm]{Corollary}
\newtheorem{lem}[thm]{Lemma}

\newtheorem{ass}[thm]{Assumption}
\newtheorem{rem}{Remark}

\theoremstyle{definition}

\numberwithin{equation}{section}

\begin{document}

\begin{frontmatter}



\title{The truncated $\theta$-Milstein method for nonautonomous and highly nonlinear stochastic differential delay equations}

\author[label1]{Shuaibin Gao}
\author[label2]{Junhao Hu}
\author[label1]{Jie He}
\author[label1]{Qian Guo\corref{cor1}}\ead{qguo@shnu.edu.cn}

\cortext[cor1]{Corresponding author at: Department of Mathematics, Shanghai Normal University, Shanghai 200234, China.}
\address[label1]{Department of Mathematics, Shanghai Normal University, Shanghai 200234, China}
\address[label2]{School of Mathematics and Statistics, South-Central University For Nationalities,
Wuhan 430074, China}

\begin{abstract}
This paper focuses on the strong convergence of the truncated $\theta$-Milstein method for a class of nonautonomous stochastic differential delay equations whose drift and diffusion coefficients can grow polynomially. The convergence rate, which is close to one, is given under the weaker assumption than the monotone condition. To verify our theoretical findings, we present a numerical example.
\end{abstract}


\end{frontmatter}



\section{Introduction}

The research on stochastic differential equations (SDEs) has been widely concerned due to their applications in numerous fields such as finance, communication, biology, chemistry, and ecology \cite{1,4,27,31}.
It is well known that time delay is widespread in nature and occurs in dynamics with a finite propagation time, then the corresponding stochastic differential delay equations (SDDEs) are widely applied  in stochastic systems \cite{6,24,26,28,35,36,37,44}.
In plenty of instances, the true solutions to SDEs cannot be expressed explicitly. Hence, it is very meaningful to simulate the true solutions with different numerical algorithms; in this way, scholars can grasp a lot of significant properties of the  true solutions without knowing the explicit form of the true solutions.
One of the most famous numerical methods for a SDE is Euler-Maruyama (EM) method, which has been investigated and developed in the past few decades \cite{20,27,31}.
Unfortunately, Hutzenthaler et al. have proved that the $p$th($p\geq 1$) moment of the EM solutions would diverge to infinity if the coefficients grow super-linearly in \cite{18}.
To approximate the stochastic equations with highly nonlinear growing coefficients, many implicit methods have been studied \cite{3,14,32,33,37,43}.
Furthermore, some modified explict schemes have been established as well, since they have less computational cost \cite{5,19,34,38}. 
Particularly, it is worth emphasizing that the truncated EM method for SDEs was initially established by Mao in \cite{29} and the convergence rate was obtained in \cite{30}. 
Since then, some scholars have studied SDEs whose coefficients can grow super-linearly by using the truncated EM method, and we can refer to \cite{8,9,10,11,13,22,25} and references therein.
To improve the convergence rate, Guo et al. \cite{12} established the truncated Milstein method (TMM) to approximate SDEs with commutative noise. 
Thereafter, the  TMM for nonautonomous  SDEs was analyzed in \cite{23}.
The convergence rate of the  TMM for SDDEs was investigated in \cite{42}.
As for other papers about Milstein methods, we refer the readers to \cite{7,16,17,21,33,38,39,45} for more detailed discussions.

The aim of this paper is to study the strong convergence rate of the truncated $\theta$-Milstein method for nonautonomous and highly nonlinear SDDEs.
The main contributions of this paper are as follows.
\begin{itemize}
	\item The SDDEs in this paper are nonautonomous, unlike many studies that focus on autonomous case. It is worth noting that the autonomy of the equation will affect the convergence rate of the numerical scheme.
	\item The truncated $\theta$-Milstein method for nonautonomous SDDEs is established in this paper, and it will degenerate into the TMM when $\theta=0$.
	\item The convergence rate in $\mathcal{L}^{2}$ sense can be obtained with the more relaxed assumptions than those  in \cite{42}, which is given in Remark \ref{remar1}. 
	\item We relax the requirements for establishing the numerical scheme, which is shown in Remark \ref{remaa6}.  
	\item The convergence rate in $\mathcal{L}^{\bar{q}}(\bar{q}\geq2)$ sense is presented in this paper, but \cite{42} only gives the convergence rate in $\mathcal{L}^{2}$ sense, which can be found in Theorems \ref{theo37} and \ref{theo311}.
\end{itemize}

This paper is organized as follows. Some necessary notations and the structure of the  truncated $\theta$-Milstein method are shown in Section 2. Main result are presented in Section 3. Section 4 contains an example. The conclusion and discussion on future research are stated in Section 5.

\section{Preliminaries}

We use the following notations throughout this paper.
Let $|x|$ denote its Euclidean norm for $x\in \mathbb{R}^n$.
Define $a\wedge b = \text {min}\{a, b\}$ and $a\vee b=\text{max}\{a,b\}$ for any real numbers $a, b$.
Let $\left \lfloor a \right \rfloor$ be the largest integer which is not greater than $a$.
Let the delay constant $\tau>0$.
Then let $\mathcal{C}([-\tau,0];\mathbb{R})$ be the family of all continuous functions $\xi$ from $[-\tau,0] $ to $\mathbb{R}$ with the norm $\|\xi\|=\text{sup}_{-\tau \le \theta \le 0}|\xi (\theta)|$.
If $S$ is a set, let $\mathbb{I}_S$ be its indicator function, which means that $\mathbb{I}_S(x)=1$ if $x \in S$ and $\mathbb{I}_S(x)=0$ if $x \notin S$.  Set $\mathbb{R} _+ = [0,+\infty)$.
Suppose that $C$ is a positive real constant which may be different in different cases.

Let $\big ( \Omega, \mathcal{F}, \{\mathcal{F}_t\}_{t\ge 0}, \mathbb{P} \big )$ stand for a complete probability space with a filtration  $\{\mathcal{F}_t\}_{t\ge 0}$ which satisfies the usual conditions (i.e., it is increasing and right continuous while $ \mathcal{F}_0$ contains all $\mathbb{P}$-null sets).
Denote by $\mathbb{E}$  the probability expectation.
Let $\mathcal{L}^p =\mathcal{L}^p( \Omega, \mathcal{F},  \mathbb{P} \big )$ for $p >0$ be the space of all random variables $X$ with $\mathbb{E}|X|^p<\infty$. Let $\mathcal{C}_{\mathcal{F}_0}^b ([-\tau,0];\mathbb{R})$ stand for the family of $\mathcal{F}_0$-measurable, $\mathcal{C}([-\tau,0];\mathbb{R})$-valued, bounded random variables.
Let $B(t) $ stand for a $1$-dimensional Brownian motion.

Consider the 1-dimensional nonautonomous and highly nonlinear SDDEs of the form
\begin{equation}\label{sdde1}
	d x(t)=f\left(t, x(t), x(t-\tau)\right) d t+g\left(t, x(t), x(t-\tau)\right) dB(t),
\end{equation}
on $t\geq 0$ with the initial data
\begin{equation}
x_{0}=\xi=\{ \xi ( \theta ) : - \tau \leq \theta \leq 0 \}\in \mathcal{C}^{b}_{\mathcal{F}_{0}}([ - \tau , 0 ];\mathbb{R}),
\end{equation}
where
$f : \mathbb{R_{+}}\times \mathbb{R}\times \mathbb{R} \rightarrow \mathbb{R}$ and  $g : \mathbb{R_{+}} \times \mathbb{R}\times \mathbb{R} \rightarrow \mathbb{R}$.

In order to study the strong convergence
rate of the truncated $\theta$-Milstein method for nonautonomous and highly nonlinear SDDEs (\ref{sdde1}), the following assumptions need to be imposed. 

\begin{ass}\label{a31}
	 There exist constants $\bar{K}>0$ and $\beta \geq 0$ such that
\begin{equation*}
	\begin{split}
	&|f(t,a,b)-f(t,\bar{a},\bar{b})|\vee|g(t,a,b)-g(t,\bar{a},\bar{b})|\\ 
	&\leq \bar{K}(1+|a|^{\beta}+|b|^{\beta}+|\bar{a}|^{\beta}+|\bar{b}|^{\beta})(|a-\bar{a}|+|b-\bar{b}|)
	\end{split}
\end{equation*}
for any $t\in [0,T]$ and $a,b,\bar{a},\bar{b}\in \mathbb{R}$.
\end{ass}

Before Assumption \ref{a32}, we give more notations. Let $\bar{\mathcal{U}}$ be the family of all continuous functions $U:\mathbb{R} \times \mathbb{R}\rightarrow \mathbb{R}_+$ such that for $\varsigma>0$, there is a constant $\rho_\varsigma >0$, which satisfies
\begin{equation}\label{zs1}
	U(m,n)\leq \rho_\varsigma|m-n|^2
\end{equation}
for any $m,n\in \mathbb{R}  $ with $|m|\vee|n|\leq \varsigma$. \\

\begin{ass}\label{a32}
	 There exist constants $K_{1}>0$ and $q>2$ such that
\begin{equation*}
\begin{split}
	 &(a-\bar{a})^{T}(f(t,a,b)-f(t,\bar{a},\bar{b}))+(q-1)|g(t,a,b)-g(t,\bar{a},\bar{b})|^{2}\\
	 &\leq K_{1}(|a-\bar{a}|^{2}+|b-\bar{b}|^{2})-U(a,\bar{a})+U(b,\bar{b})
	\end{split}
\end{equation*}
for any $t\in [0,T]$ and $a,b,\bar{a},\bar{b}\in \mathbb{R}$.
\end{ass}

\begin{rem}\label{remar1}
	An example is given to show that Assumption \ref{a32} is weaker than (A3) in \cite{42} by introducing $U(\cdot,\cdot)$. Let  $f(a,b)=\frac{1}{8}|b|^{\frac{5}{4}}-5a^{3}+2\zeta_{t}a$, $g(a,b)=\frac{1}{2}|a|^{\frac{3}{2}} +\zeta_{t}b$, $\zeta_{t}=[t(1-t)]^{\frac{3}{4}}$ for $t\in[0,1]$ and $a,b\in \mathbb{R}$. One can see that there is no $K_1$ which satisfies
	\begin{equation*}
		\begin{split}
			&(a-\bar{a})^{T}(f(t,a,b)-f(t,\bar{a},\bar{b}))+(q-1)|g(t,a,b)-g(t,\bar{a},\bar{b})|^{2}\\
			&\leq K_{1}(|a-\bar{a}|^{2}+|b-\bar{b}|^{2}),
		\end{split}
	\end{equation*}
however, Assumption \ref{a32} holds. The detailed proof can be found in Section 4.
\end{rem}

\begin{ass}\label{a33}
	There exist constants $K_{2}>0$ and $p>q$ such that
\begin{equation*}
	 a^{T}f(t,a,b)+(p-1)|g(t,a,b)|^{2}\leq K_{2}(1+|a|^{2}+|b|^{2})
\end{equation*}
for any $t\in [0,T]$ and $a,b\in \mathbb{R}$.
\end{ass}
	
\begin{ass}\label{a34}
	 There exist constants $K_{3}>0$ , $\beta \geq 0$ and $\sigma\in (0,1]$ such that
\begin{equation*}
	|f(t_{1},a,b)-f(t_{2},a,b)|\vee|g(t_{1},a,b)-g(t_{2},a,b)|\leq K_{3}(1+|a|^{\beta+1}+|b|^{\beta+1})|t_{1}-t_{2}|^{\sigma}
\end{equation*}
for any $t_{1}, t_{2}\in [0,T]$ and $a,b\in \mathbb{R}$.
\end{ass}
	
\begin{ass} \label{a35}
	There exist constants $K_{4}>0$ and $\gamma\in (0,1]$ such that
\begin{equation*}
	|\xi(t)-\xi(s)|\leq K_{4}|t-s|^{\gamma} 
\end{equation*}
for all $ t,s \in [-\tau,0]$.
\end{ass}
	
\begin{ass}\label{a36}
	 There exist constants $K_{5}>0$ and $\beta \geq 0$ such that
\begin{equation*}
	\begin{split}
		&|f_{i}(t,a,b)|\vee|g_{i}(t,a,b)|\vee|f_{ij}(t,a,b)|\vee|g_{ij}(t,a,b)|\leq K_{5}(1+|a|^{\beta+1}+|b|^{\beta+1}),
	\end{split}
\end{equation*}
where
\begin{equation}\label{notat12}
	\begin{split}
		&f_{i}(t,a_1,a_2)=\frac{\partial f(t,a_1,a_2)}{\partial a_i}, ~~~g_{i}(t,a_1,a_2)=\frac{\partial g(t,a_1,a_2)}{\partial a_i},\\
		&f_{ij}(t,a_1,a_2)=\frac{\partial ^2 f(t,a_1,a_2)}{\partial a_i\partial a_j}, ~~~g_{ij}(t,a_1,a_2)=\frac{\partial^2 g(t,a_1,a_2)}{\partial a_i \partial a_j},\\
	\end{split}
\end{equation}
for any $t \in [0,T],a,b,a_1,a_2\in \mathbb{R}$ and $i, j=1, 2$.
\end{ass}
The notations in (\ref{notat12}) will be used in the rest of this paper. The following lemma
 can be obtained with the standard method.

\begin{lem}\label{le1}
	 Let Assumptions \ref{a31}-\ref{a33} hold. 
	 There  exists a unique global solution $x(t)$ to SDDE (\ref{sdde1}), which satisfies
\begin{equation*}
	\sup _{0 \leq t \leq T} \mathbb{E}|x(t)|^{p} \leq C,~~~ \forall T \geq 0.
\end{equation*}
\end{lem}

In order to define the truncated $\theta$-Milstein method, choose a  continuous and strictly increasing function
$\Lambda:\mathbb{R} _+ \rightarrow \mathbb{R}_+$ such that $\Lambda (w) \to \infty$ when $w \rightarrow \infty$, and
\begin{equation}
	\sup_{0\leq t \leq T} \sup_{|x|\vee|y| \le w}\left ( |f(t,x,y)| \vee |g(t,x,y)| \vee |g_1(t,x,y)|\vee |g_2(t,x,y)|\right ) \le \Lambda (w),~~~\forall w\geq 1.
\end{equation}

Denote by $\Lambda ^ {-1}$ the inverse function of $\Lambda$. Therefore, $\Lambda ^ {-1}$ is a strictly increasing continuous function from $[\Lambda (1),\infty)$ to $[1,\infty)$. Next, choose $K_0 \geq [1\vee \Lambda(1)]$ and a strictly decreasing function
 $\alpha: ( 0 , 1 ] \rightarrow ( 0 , \infty )$  such that
\begin{equation}\label{eq001}
	\lim _ { \Delta \rightarrow 0 } \alpha ( \Delta ) = \infty,~~~\Delta ^{\frac{1}{4}} \alpha  (\Delta) \leq K_0,~~~ \forall \Delta\in (0,1].
\end{equation}
\begin{rem}\label{remaa6}
The condition (\ref{eq001}) is weaker than that in \cite{42}. Here, let $K_0=1$, then $\Delta \in (0,\frac{1}{\alpha  (\Delta)^4}]$. In \cite{42}, $\Delta \in (0,\frac{1}{\alpha  (\Delta)^6}]$. Obviously, the step size $\Delta$ in our paper has a wider range than that in \cite{42}, which will lead to less computation, so we have more advantages in algorithm than \cite{42}. As for other details,  please refer to Remark 2.4 in  \cite{15}.
\end{rem}
The truncated mapping $\pi_{\Delta}: \mathbb{R} \rightarrow \mathbb{R}$ for the given step size $\Delta\in (0,1]$ is defined as
\begin{equation}
	\pi_{\Delta}(\chi)=( | \chi |  \wedge \Lambda ^ {-1 }( \alpha( \Delta ) ) ) \frac { \chi } { | \chi | },
\end{equation}
where let $\chi / | \chi | = 0 $ when $ \chi = 0$.
One can see that $\pi_{\Delta}$ can map $\chi$ to itself when $| \chi |\leq \Lambda ^ {-1 }( \alpha( \Delta ) )$, and to $\Lambda ^ {-1 }( \alpha( \Delta ) )$ when $| \chi |> \Lambda ^ {-1 }( \alpha( \Delta ) )$.
Then define 
$$f_{\Delta}(t,x,y)=f(t,\pi_{\Delta}(x),\pi_{\Delta}(y)),~~~g_{\Delta}(t,x,y)=g(t,\pi_{\Delta}(x),\pi_{\Delta}(y)),$$
$$g_{1,\Delta}(t,x,y)=g_{1}(t,\pi_{\Delta}(x),\pi_{\Delta}(y)),~~~g_{2,\Delta}(t,x,y)=g_{2}(t,\pi_{\Delta}(x),\pi_{\Delta}(y)).$$

By the definition, one has that
\begin{equation}\label{eq30}
	|f_{\Delta}(t,x,y)|\vee|g_{\Delta}(t,x,y)|\vee|g_{1,\Delta}(t,x,y)|\vee|g_{2,\Delta}(t,x,y)|\leq \alpha( \Delta ).
\end{equation}

\begin{lem}\label{le2}
	By Assumption \ref{a33}, for any $\Delta\in(0,1]$, we have
	\begin{equation*}
		x^{T} f_{\Delta}(t,x,y)+(p-1)|g_{\Delta}(t,x,y)|^{2}\leq \bar{K}_{2}(1+|x|^{2}+|y|^{2})
	\end{equation*}
	for any $t\in [0,T]$ and $x,y\in \mathbb{R}$.
\end{lem}

Lemma \ref{le2} can be proved with the technique in \cite{41}, so we omit the proof process.

Let us now define the truncated $\theta$-Milstein method to approximate the solution to (\ref{sdde1}).
Suppose that  $\Delta=\frac{\tau}{M}=\frac{T}{M'}$ holds with integers $M$ and $M'$. Without losing generality, set $M<M'$.
Let $t_{k}=k\Delta$ for $k=-M,\ldots, 0 ,\ldots, M'$.
Then define

\begin{align*}
	\left \{\begin{array}{l}
		Y_{\Delta}(t_{k})=\xi(t_{k}),~~~~~~~~~~~~~~~~~~~~~~~~~~~~~~~~~~~~~~~~~~k=-M, -M+1,\ldots, 0;\\
		Y_{\Delta}(t_{k+1})=Y_{\Delta}(t_{k})+\theta f_{\Delta}(t_{k+1},Y_{\Delta} (t_{k+1}),Y_{\Delta}(t_{k+1-M}))\Delta\\
		~~~~~~+(1-\theta)f_{\Delta}(t_{k},Y_{\Delta}(t_{k}),Y_{\Delta}(t_{k-M}))\Delta
		+g_{\Delta}(t_{k},Y_{\Delta}(t_{k}),Y_{\Delta}(t_{k-M}))\Delta B_{k}\\ ~~~~~~+g_{1,\Delta}(t_{k},Y_{\Delta}(t_{k}),Y_{\Delta}(t_{k-M}))g_{\Delta}(t_{k},Y_{\Delta}(t_{k}),Y_{\Delta}(t_{k-M}))Q_{1},\\
		~~~~~~~~~~~~~~~~~~~~~~~~~~~~~~~~~~~~~~~~~~~~~~~~~~~~~~~~~~~~~~~~~~~~~k= 0, 1, \ldots, M-1;\\
	Y_{\Delta}(t_{k+1})=Y_{\Delta}(t_{k})+\theta f_{\Delta}(t_{k+1},Y_{\Delta} (t_{k+1}),Y_{\Delta}(t_{k+1-M}))\Delta\\
	~~~~~~+(1-\theta)f_{\Delta}(t_{k},Y_{\Delta}(t_{k}),Y_{\Delta}(t_{k-M}))\Delta
	+g_{\Delta}(t_{k},Y_{\Delta}(t_{k}),Y_{\Delta}(t_{k-M}))\Delta B_{k}\\ ~~~~~~+g_{1,\Delta}(t_{k},Y_{\Delta}(t_{k}),Y_{\Delta}(t_{k-M}))g_{\Delta}(t_{k},Y_{\Delta}(t_{k}),Y_{\Delta}(t_{k-M}))Q_{1}\\
	~~~~~~+g_{2,\Delta}(t_{k},Y_{\Delta}(t_{k}),Y_{\Delta}(t_{k-M}))g_{\Delta}(t_{k-M},Y_{\Delta}(t_{k-M}),Y_{\Delta}(t_{k-2M}))Q_{2},\\
	~~~~~~~~~~~~~~~~~~~~~~~~~~~~~~~~~~~~~~~~~~~~~~~~~~~~~~~~~~~~~k=M, M+1,\ldots, M'-1.
	\end{array}\right .
\end{align*}
where
\begin{equation*}
	\begin{split}
		Q_{1}=\int_{t_{k}}^{t_{k+1}}\int_{t_{k}}^{s}dB(t)dB(s)=\frac{(\Delta B_{k})^{2}-\Delta}{2},
	\end{split}	
\end{equation*}
		\begin{equation*}
			\begin{split}
		Q_{2}=\int_{t_{k}}^{t_{k+1}}\int_{t_{k}}^{s}dB(t-\tau)dB(s),~~~\Delta B_{k}=B(t_{k+1})-B(t_{k}).
	\end{split}	
\end{equation*}

There is no doubt that we only need to analyze the case when $T>\tau$, which is equivalent to $M'>M$. 
Define
\begin{equation}
	\mu(t)=\sum_{k=0}^{M'-1} t_{k} \mathbb{I}_{\left[t_{k}, t_{k+1}\right)}(t),
\end{equation}
and
\begin{equation}
	\bar{\eta}_{\Delta}(t)=\sum_{k=-M}^{M'-1} Y_{\Delta}\left(t_{k}\right) \mathbb{I}_{\left[t_{k}, t_{k+1}\right)}(t).
\end{equation}

The continuous form of the truncated $\theta$-Milstein method can be defined as
\begin{equation}\label{eq21}
	\begin{split}
		&\eta_{\Delta}(t)-\theta f_{\Delta}(t,\eta_{\Delta}(t),\eta_{\Delta}(t-\tau))\Delta\\
		&=\xi(0)-\theta f_{\Delta}(0,\xi(0),\xi(-\tau))\Delta
		+\int_{0}^{t}f_{\Delta}(\mu(s),\bar{\eta}_{\Delta}(s),\bar{\eta}_{\Delta}(s-\tau))ds
		\\&~~~+\int_{0}^{t}g_{\Delta}(\mu(s),\bar{\eta}_{\Delta}(s),\bar{\eta}_{\Delta}(s-\tau))dB(s)
		\\&~~~+\int_{0}^{t}g_{1,\Delta}(\mu(s),\bar{\eta}_{\Delta}(s),\bar{\eta}_{\Delta}(s-\tau))
		g_{\Delta}(\mu(s),\bar{\eta}_{\Delta}(s),\bar{\eta}_{\Delta}(s-\tau))\Delta \hat{B}(s)dB(s)
		\\&~~~+\int_{0}^{t}g_{2,\Delta}(\mu(s),\bar{\eta}_{\Delta}(s),\bar{\eta}_{\Delta}(s-\tau))\\
		&~~~~~~\cdot g_{\Delta}(\mu(s-\tau),\bar{\eta}_{\Delta}(s-\tau),\bar{\eta}_{\Delta}(s-2\tau))\Delta \hat{B}(s-\tau)dB(s),
	\end{split}
\end{equation}
 where $\Delta \hat{B}(t)=\sum_{k=-M}^{M'-1}\left(B(t)-B\left(t_{k}\right)\right) \mathbb{I}_{\left[ t_{k} , t_{k+1}\right)}(t)$.

By the monotone operator theory in \cite{40}, there exists a unique $Y_\Delta (t_{k+1})$ for the given $Y_\Delta (t_{k})$ when $K_1 \theta \Delta <1$ holds. Let $\Delta^*=1\wedge \frac{1}{K_1 \theta}$. 
In the rest of this paper, let $\Delta\in \left(0, \Delta^{*}\right)$ and $\theta\in (0,1]$.
To simplify the notations, set $\kappa(t)=\lfloor t/\Delta\rfloor \Delta$ for $t \in [-\tau, T]$. Additionally, define
\begin{equation}\label{eq22}
	Z_{\Delta}(t)=\eta_{\Delta}(t)-\theta f_{\Delta}(t,\eta_{\Delta}(t),\eta_{\Delta}(t-\tau))\Delta,
\end{equation}
\begin{equation}\label{eq23}
	\bar{Z}_{\Delta}(t)=\bar{\eta}_{\Delta}(t)-\theta f_{\Delta}(\mu(t),\bar{\eta}_{\Delta}(t),\bar{\eta}_{\Delta}(t-\tau))\Delta.
\end{equation}

The following Taylor expansion would play a significant role in our proof. For more details, please refer to \cite{2}.

If $\psi:\mathbb{R}^{3} \rightarrow \mathbb{R}$ is a third-order continuously differentiable function, one can see that
\begin{equation}
	\psi(\bar{u})-\psi(\hat{u})=\psi'(u)\mid _{u=\hat{u}}(\bar{u}-\hat{u})+R_{\psi}(\bar{u}, \hat{u}),
\end{equation}
where $$R_{\psi}(\bar{u}, \hat{u})=\int_{0}^{1}(1-\vartheta)\psi''(u)\mid _{u=\hat{u}+\vartheta(\bar{u}-\hat{u})}(\bar{u}-\hat{u}, \bar{u}-\hat{u})d\vartheta,$$
for any $u, \bar{u},\hat{u}\in \mathbb{R}^{3}$. Here, $\psi'$ and $\psi''$ are defined as
\begin{equation}
	\psi'(u)(j)=\sum_{i=1}^{3}\frac{\partial\psi}{\partial u_{i}}j_{i},~~~
	\psi''(u)(j,h)=\sum_{i,m=1}^{3}\frac{\partial^{2}\psi}{\partial u_{i}\partial u_{m}}j_{i}h_{m}.
\end{equation}
where $j=(j_{1},j_{2},j_{3}), h=(h_{1},h_{2},h_{3})$.

Next, set $\bar{u}=(\tilde{\mu}, \bar{x}, \bar{y})$ and $\hat{u}=(\tilde{\mu}, \hat{x}, \hat{y})$ for $\bar{x}, \bar{y}, \hat{x}, \hat{y}\in \mathbb{R}$, $\tilde{\mu}\in \mathbb{R_{+}}$. Thus, $\bar{u}-\hat{u}=(0,\bar{x}- \hat{x},\bar{y}-\hat{y})$.
Then, for $\psi:\mathbb{R_{+}}\times \mathbb{R}\times \mathbb{R} \rightarrow \mathbb{R}$, one can see that 
\begin{equation}\label{eqqq1}
	\psi(\tilde{\mu}, \bar{x}, \bar{y})-\psi(\tilde{\mu}, \hat{x}, \hat{y})=\psi'(\tilde{\mu},x,y)\mid _{x=\hat{x},y=\hat{y}}(0,\bar{x}- \hat{x},\bar{y}-\hat{y})+R_{\psi}(\tilde{\mu},\bar{x}, \bar{y}, \hat{x}, \hat{y}),
\end{equation}
where 
\begin{equation}\label{eq32}
	\begin{split}
R_{\psi}(\tilde{\mu},\bar{x}, \bar{y}, \hat{x}, \hat{y})=
\int_{0}^{1}&(1-\vartheta)\psi''(\tilde{\mu},x,y)\mid _{x=\hat{x}+\vartheta(\bar{x}-\hat{x}), y=\hat{y}+\vartheta(\bar{y}-\hat{y})}\\
&\big((0,\bar{x}-\hat{x}, \bar{y}-\hat{y}), (0,\bar{x}-\hat{x}, \bar{y}-\hat{y})\big)d\vartheta,
\end{split}
\end{equation}	
for any $\tilde{\mu}\in \mathbb{R_{+}}$ and $\bar{x}, \bar{y}, \hat{x}, \hat{y}\in \mathbb{R}$. Here, $\psi'$ and $\psi''$ are defined as
\begin{equation*}
	\psi'(\tilde{\mu},x,y)(0,\bar{x}-\hat{x}, \bar{y}-\hat{y})=\frac{\partial\psi}{\partial x}(\bar{x}-\hat{x})+\frac{\partial\psi}{\partial y}(\bar{y}-\hat{y}),
\end{equation*}
\begin{equation*}
	\begin{split}
		\psi''&(\tilde{\mu},x,y)\big((0,\bar{x}-\hat{x}, \bar{y}-\hat{y}),(0,\bar{x}-\hat{x}, \bar{y}-\hat{y})\big)
		\\=&\frac{\partial^{2}\psi}{\partial x\partial x}(\bar{x}-\hat{x})^{2}+\frac{\partial^{2}\psi}{\partial x\partial y}(\bar{x}-\hat{x})(\bar{y}-\hat{y})
		+\frac{\partial^{2}\psi}{\partial y\partial y}(\bar{y}-\hat{y})^{2}+\frac{\partial^{2}\psi}{\partial y\partial x}(\bar{y}-\hat{y})(\bar{x}-\hat{x}).
	\end{split}
\end{equation*}
for  $\tilde{\mu}\in \mathbb{R_{+}}$ and $\bar{x}, \bar{y}, \hat{x}, \hat{y}\in \mathbb{R}$.
By setting $\tilde{\mu}=\mu(t), \bar{x}=\eta_{\Delta}(t), \bar{y}=\eta_{\Delta}(t-\tau), \hat{x}=\bar{\eta}_{\Delta}(t), \hat{y}=\bar{\eta}_{\Delta}(t-\tau)$, we get from (\ref{eqqq1}) that
\begin{equation}\label{eqeq07}
	\begin{split}
		&\psi(\mu(t),\eta_{\Delta}(t), \eta_{\Delta}(t-\tau))-\psi(\mu(t),\bar{\eta}_{\Delta}(t), \bar{\eta}_{\Delta}(t-\tau))\\
		&=\psi_{1}(\mu(t),\bar{\eta}_{\Delta}(t), \bar{\eta}_{\Delta}(t-\tau))(\eta_{\Delta}(t)-\bar{\eta}_{\Delta}(t))
		\\
		&+\psi_{2}(\mu(t),\bar{\eta}_{\Delta}(t), \bar{\eta}_{\Delta}(t-\tau))(\eta_{\Delta}(t-\tau)-\bar{\eta}_{\Delta}(t-\tau))
		\\
		&+R_{\psi}(t,\eta_{\Delta}(t), \bar{\eta}_{\Delta}(t), \eta_{\Delta}(t-\tau),\bar{\eta}_{\Delta}(t-\tau))
		\\&=\psi_{1}(\mu(t),\bar{\eta}_{\Delta}(t), \bar{\eta}_{\Delta}(t-\tau))\int_{\kappa(t)}^{t}g_{\Delta}(\mu(s),\bar{\eta}_{\Delta}(s), \bar{\eta}_{\Delta}(s-\tau))dB(s)
		\\&+\psi_{2}(\mu(t),\bar{\eta}_{\Delta}(t), \bar{\eta}_{\Delta}(t-\tau))\int_{\kappa(t)-\tau}^{t-\tau}g_{\Delta}(\mu(s),\bar{\eta}_{\Delta}(s), \bar{\eta}_{\Delta}(s-\tau))dB(s)
		\\&+\bar{R}_{\psi}(t,\eta_{\Delta}(t), \bar{\eta}_{\Delta}(t), \eta_{\Delta}(t-\tau),\bar{\eta}_{\Delta}(t-\tau)),
	\end{split}
\end{equation}
where
\begin{equation}\label{eq33}
	\begin{split}
		\bar{R}_{\psi}&(t,\eta_{\Delta}(t), \bar{\eta}_{\Delta}(t), \eta_{\Delta}(t-\tau),\bar{\eta}_{\Delta}(t-\tau))
		\\=&\psi_{1}(\mu(t),\bar{\eta}_{\Delta}(t), \bar{\eta}_{\Delta}(t-\tau))\\
		&\cdot \Big[\theta f_{\Delta}(t,\eta_{\Delta}(t), \eta_{\Delta}(t-\tau))\Delta-\theta f_{\Delta}(\mu(t),\bar{\eta}_{\Delta}(t), \bar{\eta}_{\Delta}(t-\tau))\Delta
		\\&\left.+\int_{\kappa(t)}^{t}f_{\Delta}(\mu(s),\bar{\eta}_{\Delta}(s), \bar{\eta}_{\Delta}(s-\tau))ds\right.
		\\&\left.+\int_{\kappa(t)}^{t}g_{1,\Delta}(\mu(s),\bar{\eta}_{\Delta}(s), \bar{\eta}_{\Delta}(s-\tau))g_{\Delta}(\mu(s),\bar{\eta}_{\Delta}(s), \bar{\eta}_{\Delta}(s-\tau))\Delta \hat{B}(s)dB(s)\right.
		\\&+\int_{\kappa(t)}^{t}g_{2,\Delta}(\mu(s),\bar{\eta}_{\Delta}(s), \bar{\eta}_{\Delta}(s-\tau))\\
		&\cdot g_{\Delta}(\mu(s-\tau),\bar{\eta}_{\Delta}(s-\tau), \bar{\eta}_{\Delta}(s-2\tau))\Delta \hat{B}(s-\tau)dB(s)\Big]
		\\&+\psi_{2}(\mu(t),\bar{\eta}_{\Delta}(t), \bar{\eta}_{\Delta}(t-\tau))\\
		&\cdot \Big[\theta f_{\Delta}(t-\tau,\eta_{\Delta}(t-\tau), \eta_{\Delta}(t-2\tau))\Delta-\theta f_{\Delta}(\mu(t-\tau),\bar{\eta}_{\Delta}(t-\tau), \bar{\eta}_{\Delta}(t-2\tau))\Delta
		\\&\left.+\int_{\kappa(t)-\tau}^{t-\tau}f_{\Delta}(\mu(s),\bar{\eta}_{\Delta}(s), \bar{\eta}_{\Delta}(s-\tau))ds\right.
		\\&\left.+\int_{\kappa(t)-\tau}^{t-\tau}g_{1,\Delta}(\mu(s),\bar{\eta}_{\Delta}(s), \bar{\eta}_{\Delta}(s-\tau))g_{\Delta}(\mu(s),\bar{\eta}_{\Delta}(s), \bar{\eta}_{\Delta}(s-\tau))\Delta \hat{B}(s)dB(s)\right.
		\\&+\int_{\kappa(t)-\tau}^{t-\tau}g_{2,\Delta}(\mu(s),\bar{\eta}_{\Delta}(s), \bar{\eta}_{\Delta}(s-\tau))\\
		&\cdot g_{\Delta}(\mu(s-\tau),\bar{\eta}_{\Delta}(s-\tau), \bar{\eta}_{\Delta}(s-2\tau))\Delta \hat{B}(s-\tau)dB(s)\Big]
		\\&+ R _{\psi}(\mu(t),\eta_{\Delta}(t), \bar{\eta}_{\Delta}(t), \eta_{\Delta}(t-\tau),\bar{\eta}_{\Delta}(t-\tau)),
	\end{split}
\end{equation}
and
\begin{equation}
	\psi_{1}(t,x,y)=\frac{\partial\psi}{\partial x},~~~ \psi_{2}(t,x,y)=\frac{\partial\psi}{\partial y}.
\end{equation}

\section{Strong convergence rate}

For ${\bar{q}}\geq 2$, the strong convergence rate in $\mathcal{L}^{\bar{q}}$ sense is investigated in this section. 
First,  some necessary lemmas are presented.

\begin{lem}\label{le31}
	For any $\Delta\in (0,\Delta^{*})$ and $t\in [0,T]$, one can see that
	\begin{equation}\label{mea1}
		\mathbb{E}\left|Z_{\Delta}(t)-\bar{Z}_{\Delta}(t)\right|^{\bar{p}}\leq C\Delta^{\frac{\bar{p}}{2}}\alpha(\Delta)^{\bar{p}},~~~\forall \bar{p} >0,
	\end{equation}
and
\begin{equation}\label{mea2}
	\mathbb{E}\left|\eta_{\Delta}(t)-\bar{\eta}_{\Delta}(t)\right|^{\bar{p}}\leq C\Delta^{\frac{\bar{p}}{2}}\alpha(\Delta)^{\bar{p}},~~~\forall \bar{p} >0.
\end{equation}
Thus,
\begin{equation*}
	\lim_{\Delta\rightarrow 0}\mathbb{E}\left|Z_{\Delta}(t)-\bar{Z}_{\Delta}(t)\right|^{\bar{p}}=\lim_{\Delta\rightarrow 0}\mathbb{E}\left|\eta_{\Delta}(t)-\bar{\eta}_{\Delta}(t)\right|^{\bar{p}}=0,~~~\forall \bar{p} >0.
\end{equation*}
\end{lem}
\begin{proof}
	By (\ref{eq21}), (\ref{eq22}) and (\ref{eq23}), we derive that
	\begin{equation}\label{eq25}
		\begin{split}
			Z_{\Delta}&(t)=\bar{Z}_{\Delta}(t)+\int_{\kappa(t)}^{t}f_{\Delta}(\mu(s),\bar{\eta}_{\Delta}(s), \bar{\eta}_{\Delta}(s-\tau))ds\\
			&+\int_{\kappa(t)}^{t}g_{\Delta}(\mu(s),\bar{\eta}_{\Delta}(s), \bar{\eta}_{\Delta}(s-\tau))dB(s)
			\\&+\int_{\kappa(t)}^{t}g_{1,\Delta}(\mu(s),\bar{\eta}_{\Delta}(s), \bar{\eta}_{\Delta}(s-\tau)) g_{\Delta}(\mu(s),\bar{\eta}_{\Delta}(s), \bar{\eta}_{\Delta}(s-\tau))\Delta \hat{B}(s)dB(s)
			\\&+\int_{\kappa(t)}^{t}g_{2,\Delta}(\mu(s),\bar{\eta}_{\Delta}(s), \bar{\eta}_{\Delta}(s-\tau))\\
			&\cdot g_{\Delta}(\mu(s-\tau),\bar{\eta}_{\Delta}(s-\tau), \bar{\eta}_{\Delta}(s-2\tau))\Delta \hat{B}(s-\tau)dB(s).
		\end{split}
	\end{equation}
	For any fixed $\bar{p}\geq2$ and $t\in [0,T]$, we can get from the Burkholder-Davis-Gundy	inequality and the  H\"{o}lder  inequality that
	\begin{equation*}
		\begin{split}
			&\mathbb{E}\left|Z_{\Delta}(t)-\bar{Z}_{\Delta}(t)\right|^{\bar{p}}\\
			\\&\leq C\Big[\Delta^{\bar{p}-1}\mathbb{E}\int_{\kappa(t)}^{t}|f_{\Delta}(\mu(s),\bar{\eta}_{\Delta}(s), \bar{\eta}_{\Delta}(s-\tau))|^{\bar{p}}ds
			\\&\left.\quad+\Delta^{\frac{\bar{p}}{2}-1}\mathbb{E}\int_{\kappa(t)}^{t}|g_{\Delta}(\mu(s),\bar{\eta}_{\Delta}(s), \bar{\eta}_{\Delta}(s-\tau))|^{\bar{p}}ds\right.
			\\&\left.\quad+\Delta^{\frac{\bar{p}}{2}-1}\mathbb{E}\int_{\kappa(t)}^{t}|g_{1,\Delta}(\mu(s),\bar{\eta}_{\Delta}(s), \bar{\eta}_{\Delta}(s-\tau))g_{\Delta}(\mu(s),\bar{\eta}_{\Delta}(s), \bar{\eta}_{\Delta}(s-\tau))\Delta \hat{B}(s)|^{\bar{p}}ds\right.
		\end{split}
	\end{equation*}
			\begin{equation*}
				\begin{split}
			&+\Delta^{\frac{\bar{p}}{2}-1}\mathbb{E}\int_{\kappa(t)}^{t}|g_{2,\Delta}(\mu(s),\bar{\eta}_{\Delta}(s), \bar{\eta}_{\Delta}(s-\tau))\\
			&\quad\quad \cdot g_{\Delta}(\mu(s-\tau),\bar{\eta}_{\Delta}(s-\tau), \bar{\eta}_{\Delta}(s-2\tau))\Delta \hat{B}(s-\tau))|^{\bar{p}}ds
			\Big]
			\\&\leq C\left(\Delta^{\bar{p}}\alpha(\Delta)^{\bar{p}}
			+\Delta^{\frac{\bar{p}}{2}}\alpha(\Delta)^{\bar{p}}
			+
			\Delta^{\bar{p}}\alpha(\Delta)^{2\bar{p}}\right)
			\leq C\Delta^{\frac{\bar{p}}{2}}\alpha(\Delta)^{\bar{p}},
		\end{split}
	\end{equation*}
where $\Delta^{\frac{\bar{p}}{2}}\alpha(\Delta)^{\bar{p}}\leq K_0^{\bar{p}}\Delta^{\frac{\bar{p}}{4}}$ is used.
	For $0<\bar{p}<2$, applying H\"{o}lder's inequality can give the desired result (\ref{mea1}). Then combining (\ref{eq22}),  (\ref{eq23}) and  (\ref{mea1}) yields (\ref{mea2}). The proof is complete.
\end{proof}

\begin{lem}\label{le32}
	Let Assumptions \ref{a31} and \ref{a33} hold. Then we derive that
	\begin{equation*}
		\sup _{0 < \Delta < \Delta^{*}}\sup _{0 \leq t \leq T} \mathbb{E}\left|\eta_{\Delta}(t)\right|^{p}\leq C.
	\end{equation*}
\end{lem}
\begin{proof}
	We get from It\^{o}'s formula  and (\ref{eq25}) that
	\begin{equation*}
		\begin{split}
			&\mathbb{E}|Z_{\Delta}(t)|^{p}-|Z_{\Delta}(0)|^{p}\\
			&\leq
			\mathbb{E}\int_{0}^{t}p|Z_{\Delta}(s)|^{p-2}\left( Z_{\Delta}(s)^T f_{\Delta}(\mu(s),\bar{\eta}_{\Delta}(s), \bar{\eta}_{\Delta}(s-\tau))\right.+\frac{p-1}{2}|g_{\Delta}(\mu(s),\bar{\eta}_{\Delta}(s), \bar{\eta}_{\Delta}(s-\tau))
			\\&\left.+g_{1,\Delta}(\mu(s),\bar{\eta}_{\Delta}(s), \bar{\eta}_{\Delta}(s-\tau))g_{\Delta}(\mu(s),\bar{\eta}_{\Delta}(s), \bar{\eta}_{\Delta}(s-\tau))\Delta \hat{B}(s)\right.
			\\&\left.+g_{2,\Delta}(\mu(s),\bar{\eta}_{\Delta}(s), \bar{\eta}_{\Delta}(s-\tau))g_{\Delta}(\mu(s-\tau),\bar{\eta}_{\Delta}(s-\tau), \bar{\eta}_{\Delta}(s-2\tau))\Delta \hat{B}(s-\tau)|^{2}\right)ds
			\\&\leq\mathbb{E}\int_{0}^{t}p|Z_{\Delta}(s)|^{p-2}( Z_{\Delta}(s)-\bar{\eta}_{\Delta}(s))^T f_{\Delta}(\mu(s),\bar{\eta}_{\Delta}(s), \bar{\eta}_{\Delta}(s-\tau)) ds
			\\&+\mathbb{E}\int_{0}^{t}p|Z_{\Delta}(s)|^{p-2}( \bar{\eta}_{\Delta}(s)^T f_{\Delta}(\mu(s),\bar{\eta}_{\Delta}(s), \bar{\eta}_{\Delta}(s-\tau))+(p-1)|g_{\Delta}(\mu(s),\bar{\eta}_{\Delta}(s), \bar{\eta}_{\Delta}(s-\tau))|^{2}) ds
			\\&+\mathbb{E}\int_{0}^{t}p(p-1)|Z_{\Delta}(s)|^{p-2}\left|g_{1,\Delta}(\mu(s),\bar{\eta}_{\Delta}(s), \bar{\eta}_{\Delta}(s-\tau))g_{\Delta}(\mu(s),\bar{\eta}_{\Delta}(s), \bar{\eta}_{\Delta}(s-\tau))\Delta \hat{B}(s)\right.
			\\&\left.+g_{2,\Delta}(\mu(s),\bar{\eta}_{\Delta}(s), \bar{\eta}_{\Delta}(s-\tau))g_{\Delta}(\mu(s-\tau),\bar{\eta}_{\Delta}(s-\tau), \bar{\eta}_{\Delta}(s-2\tau))\Delta \hat{B}(s-\tau)\right|^{2} ds
			\\&=:I_{1}+I_{2}+I_{3}.
		\end{split}
	\end{equation*}
By Lemma \ref{le31}, we derive that
	\begin{equation}\label{iii1}
		\begin{split}
			I_{1}\leq&(p-2)\mathbb{E}\int_{0}^{t}|Z_{\Delta}(s)|^{p}ds+2\mathbb{E}\int_{0}^{t}|( Z_{\Delta}(s)-\bar{\eta}_{\Delta}(s))^Tf_{\Delta}(\mu(s),\bar{\eta}_{\Delta}(s), \bar{\eta}_{\Delta}(s-\tau))|^{\frac{p}{2}}ds
			\\\leq& C\mathbb{E}\int_{0}^{t}|Z_{\Delta}(s)|^{p}ds+C\mathbb{E}\int_{0}^{t}|( Z_{\Delta}(s)-\bar{Z}_{\Delta}(s))^Tf_{\Delta}(\mu(s),\bar{\eta}_{\Delta}(s), \bar{\eta}_{\Delta}(s-\tau))|^{\frac{p}{2}}ds
			\\&+C\mathbb{E}\int_{0}^{t}|\theta f_{\Delta}^{2}(\mu(s),\bar{\eta}_{\Delta}(s), \bar{\eta}_{\Delta}(s-\tau))\Delta|^{\frac{p}{2}}ds
			\\\leq& C \mathbb{E}\int_{0}^{t}|Z_{\Delta}(s)|^{p}ds+C\mathbb{E}\int_{0}^{t}\Delta^{\frac{p}{4}}\alpha(\Delta)^{p}ds
			+C\mathbb{E}\int_{0}^{t}\theta^{\frac{p}{2}}\Delta^{\frac{p}{2}}\alpha(\Delta)^{p} ds
			\\\leq& C+C\mathbb{E}\int_{0}^{t}|Z_{\Delta}(s)|^{p}ds.
		\end{split}
	\end{equation}
By Lemma \ref{le2}, we get that
	\begin{equation}\label{iii2}
		\begin{split}
			I_{2}\leq&C\mathbb{E}\int_{0}^{t}|Z_{\Delta}(s)|^{p-2}(1+|\bar{\eta}_{\Delta}(s)|^{2}+|\bar{\eta}_{\Delta}(s-\tau)|^{2})ds
			\\\leq&C\mathbb{E}\int_{0}^{t}|Z_{\Delta}(s)|^{p}ds+C\mathbb{E}\int_{0}^{t}(1+|\bar{\eta}_{\Delta}(s)|^{p}+|\bar{\eta}_{\Delta}(s-\tau)|^{p})ds.
		\end{split}
	\end{equation}
Using (\ref{eq001}) and (\ref{eq30}) yields that
	\begin{equation}\label{iii3}
		\begin{split}
			I_{3}\leq&C\mathbb{E}\int_{0}^{t}|Z_{\Delta}(s)|^{p}ds\\
			&+C\mathbb{E}\int_{0}^{t}\left|g_{1,\Delta}(\mu(s),\bar{\eta}_{\Delta}(s), \bar{\eta}_{\Delta}(s-\tau))g_{\Delta}(\mu(s),\bar{\eta}_{\Delta}(s), \bar{\eta}_{\Delta}(s-\tau))\Delta \hat{B}(s)\right.
			\\&\left.+g_{2,\Delta}(\mu(s),\bar{\eta}_{\Delta}(s), \bar{\eta}_{\Delta}(s-\tau))g_{\Delta}(\mu(s-\tau),\bar{\eta}_{\Delta}(s-\tau), \bar{\eta}_{\Delta}(s-2\tau))\Delta \hat{B}(s-\tau))\right|^{p}ds
			\\\leq&C\mathbb{E}\int_{0}^{t}|Z_{\Delta}(s)|^{p}ds+C\mathbb{E}\int_{0}^{t}\Delta^{\frac{p}{2}}\alpha(\Delta)^{2p} ds
			\\\leq&C+C\mathbb{E}\int_{0}^{t}|Z_{\Delta}(s)|^{p}ds.
		\end{split}
	\end{equation}
	Thus, from (\ref{iii1}) - (\ref{iii3}), we have
	\begin{equation}
		\mathbb{E}|Z_{\Delta}(t)|^{p}\leq C+C\int_{0}^{t}\mathbb{E}|Z_{\Delta}(s)|^{p}ds+C\mathbb{E}\int_{0}^{t}\left(|\bar{\eta}_{\Delta}(s)|^{p}+|\bar{\eta}_{\Delta}(s-\tau)|^{p}\right)ds.
	\end{equation}
	Applying Gronwall's inequality gives that
	\begin{equation*}
		\mathbb{E}|Z_{\Delta}(t)|^{p}\leq C+C\mathbb{E}\int_{0}^{t}\left(|\bar{\eta}_{\Delta}(s)|^{p}+|\bar{\eta}_{\Delta}(s-\tau)|^{p}\right)ds.
	\end{equation*}
	By (\ref{eq22}) and the elementary inequality
	$|u-v|^{p}\geq2^{1-p}|u|^{p}-|v|^{p},  u,v>0,$
	one can see that
	\begin{equation*}
		\begin{split}
			|Z_{\Delta}(t)|^{p}\geq&2^{1-p}|\eta_{\Delta}(t)|^{p}-\theta^{p}\Delta^{p}|f_{\Delta}(t,\eta_{\Delta}(t), \eta_{\Delta}(t-\tau))|^{p}
			\\\geq&2^{1-p}|\eta_{\Delta}(t)|^{p}-\theta^{p}\Delta^{p}\alpha (\Delta)^{p}.
		\end{split}
	\end{equation*}
	Thus,
	\begin{equation*}
		\begin{split}
			\sup_{0\leq r\leq t}\mathbb{E}|\eta_{\Delta}(r)|^{p}\leq C\left(1+\int_{0}^{t}\sup_{0\leq r\leq s}\mathbb{E}|Z_{\Delta}(r)|^{p}ds\right)
			\leq C\left(1+\int_{0}^{t}\sup_{0\leq r\leq s}\mathbb{E}|\eta_{\Delta}(r)|^{p}ds\right).
		\end{split}
	\end{equation*}
	Using Gronwall's inequality again yields that
	\begin{equation*}
			\sup _{0 < \Delta < \Delta^{*}}\sup _{0 \leq t \leq T} \mathbb{E}\left|\eta_{\Delta}(t)\right|^{p}\leq C.
	\end{equation*}
We complete the proof.
\end{proof}

The following lemma could be obtained by using  Lemma \ref{le32}.
\begin{lem}\label{lemmaa33}
	Let Assumptions \ref{a31}, \ref{a33} and \ref{a36} hold. For  $q>2$,
suppose that $p\geq2(1+\beta)q$ holds.  Then for any $\bar{q}\in [2,q)$, $\Delta\in (0, \Delta^{*})$ and $t\in [0,T]$, we get that
	\begin{equation*}
		\begin{split}
			\sup _{0 < \Delta < \Delta^{*}}&\sup _{0 \leq t \leq T}\left(\mathbb{E}|f(t,\eta_{\Delta}(t), \eta_{\Delta}(t-\tau))|^{2\bar{q}}\vee\mathbb{E}|g(t,\eta_{\Delta}(t), \eta_{\Delta}(t-\tau))|^{2\bar{q}}\right.
			\\&\vee \mathbb{E}|f_{i}(t,\eta_{\Delta}(t), \eta_{\Delta}(t-\tau))|^{2\bar{q}}
		\vee\mathbb{E}|g_{i}(t,\eta_{\Delta}(t), \eta_{\Delta}(t-\tau))|^{2\bar{q}}\\
			&	\left.\vee\mathbb{E}|f_{ij}(t,\eta_{\Delta}(t), \eta_{\Delta}(t-\tau))|^{2\bar{q}}\vee\mathbb{E}|g_{ij}(t,\eta_{\Delta}(t), \eta_{\Delta}(t-\tau))|^{2\bar{q}}
			\right)\leq C.
		\end{split}
	\end{equation*}
\end{lem}

\begin{lem}\label{le34}
	Let Assumptions  \ref{a31}, \ref{a33} and \ref{a36} hold. For  $q>2$,
	suppose that $p\geq2(1+\beta)q$ holds. Then for any $\bar{q}\in [2,q)$, $\Delta\in (0, \Delta^{*})$ and $t\in [0,T]$, we have
	\begin{equation*}
		\begin{split}
			&\quad\mathbb{E}|R_{f}(t,\eta_{\Delta}(t),\bar{\eta}_{\Delta}(t) ,\eta_{\Delta}(t-\tau),\bar{\eta}_{\Delta}(t-\tau))|^{\bar{q}}\\
			&\vee\mathbb{E}|R_{g}(t,\eta_{\Delta}(t),\bar{\eta}_{\Delta}(t) ,\eta_{\Delta}(t-\tau),\bar{\eta}_{\Delta}(t-\tau))|^{\bar{q}}
			\\&\vee\mathbb{E}|\bar{R}_{f}(t,\eta_{\Delta}(t),\bar{\eta}_{\Delta}(t) ,\eta_{\Delta}(t-\tau),\bar{\eta}_{\Delta}(t-\tau))|^{\bar{q}}
			\\&\vee\mathbb{E}|\bar{R}_{g}(t,\eta_{\Delta}(t),\bar{\eta}_{\Delta}(t) ,\eta_{\Delta}(t-\tau),\bar{\eta}_{\Delta}(t-\tau))|^{\bar{q}}
		\leq C\Delta^{\bar{q}}\alpha(\Delta)^{2\bar{q}}.
		\end{split}
	\end{equation*}
\end{lem}
\begin{proof}
	For any $\bar{q}\in [2,q)$, $\Delta\in (0, \Delta^{*})$ and $t\in [0,T]$, we get from H\"{o}lder's inequality, Lemma \ref{le32} and (\ref{eq32}) that
	\begin{equation*}
		\begin{split}
			&\mathbb{E}|R_{f}(t,\eta_{\Delta}(t),\bar{\eta}_{\Delta}(t) ,\eta_{\Delta}(t-\tau),\bar{\eta}_{\Delta}(t-\tau))|^{\bar{q}}
			\\&\leq\int_{0}^{1}(1-\vartheta)^{\bar{q}}\mathbb{E}\left|f''(\mu(t),x,y)\mid _{x=\bar{\eta}_{\Delta}(t)+\vartheta(\eta_{\Delta}(t)-\bar{\eta}_{\Delta}(t)),
				y=\bar{\eta}_{\Delta}(t-\tau)+\vartheta(\eta_{\Delta}(t-\tau)-\bar{\eta}_{\Delta}(t-\tau))}
			\right.\\&\quad\quad((0,\eta_{\Delta}(t)-\bar{\eta}_{\Delta}(t),\eta_{\Delta}(t-\tau)-\bar{\eta}_{\Delta}(t-\tau)),\\
			&\quad \quad\quad \left.(0,\eta_{\Delta}(t)-\bar{\eta}_{\Delta}(t),\eta_{\Delta}(t-\tau)-\bar{\eta}_{\Delta}(t-\tau)))\right|^{\bar{q}}d\vartheta
			\\&\leq\int_{0}^{1}\left[(\mathbb{E}|f_{11}(\mu(t),\bar{\eta}_{\Delta}(t)+\vartheta(\eta_{\Delta}(t)-\bar{\eta}_{\Delta}(t)),\bar{\eta}_{\Delta}(t-\tau)+\vartheta(\eta_{\Delta}(t-\tau)-\bar{\eta}_{\Delta}(t-\tau)))|^{2\bar{q}}\right.\\
			&\quad\quad\quad\cdot\mathbb{E}|\eta_{\Delta}(t)
			-\bar{\eta}_{\Delta}(t)|^{4\bar{q}})^{\frac{1}{2}}
			\\&\quad+(\mathbb{E}|f_{12}(\mu(t),\bar{\eta}_{\Delta}(t)+\vartheta(\eta_{\Delta}(t)-\bar{\eta}_{\Delta}(t)),\bar{\eta}_{\Delta}(t-\tau)+\vartheta(\eta_{\Delta}(t-\tau)-\bar{\eta}_{\Delta}(t-\tau)))|^{2\bar{q}}\\
			&\quad\quad\quad\cdot\mathbb{E}(|\eta_{\Delta}(t)
			-\bar{\eta}_{\Delta}(t)||\eta_{\Delta}(t-\tau)
			-\bar{\eta}_{\Delta}(t-\tau)|)^{2\bar{q}})^{\frac{1}{2}}
			\\&\quad+(\mathbb{E}|f_{21}(\mu(t),\bar{\eta}_{\Delta}(t)+\vartheta(\eta_{\Delta}(t)-\bar{\eta}_{\Delta}(t)),\bar{\eta}_{\Delta}(t-\tau)+\vartheta(\eta_{\Delta}(t-\tau)-\bar{\eta}_{\Delta}(t-\tau)))|^{2\bar{q}}\\
			&\quad\quad\quad\cdot\mathbb{E}(|\eta_{\Delta}(t)
			-\bar{\eta}_{\Delta}(t)||\eta_{\Delta}(t-\tau)
			-\bar{\eta}_{\Delta}(t-\tau)|)^{2\bar{q}})^{\frac{1}{2}}
			\\&\quad+(\mathbb{E}|f_{22}(\mu(t),\bar{\eta}_{\Delta}(t)+\vartheta(\eta_{\Delta}(t)-\bar{\eta}_{\Delta}(t)),\bar{\eta}_{\Delta}(t-\tau)+\vartheta(\eta_{\Delta}(t-\tau)-\bar{\eta}_{\Delta}(t-\tau)))|^{2\bar{q}}\\
			&\quad\quad\quad\left.\cdot\mathbb{E}|\eta_{\Delta}(t-\tau)
			-\bar{\eta}_{\Delta}(t-\tau)|^{4\bar{q}})^{\frac{1}{2}}\right]d\vartheta
			\\&\leq C\left(1+\sup _{0 < \Delta < \Delta^{*}}\sup _{0 \leq t \leq T}\mathbb{E}|\eta_{\Delta}(t)|^{2\bar{q}(1+\beta)}\right)^{\frac{1}{2}}\left(\Delta^{2\bar{q}}\alpha(\Delta)^{4\bar{q}}\right)^{\frac{1}{2}}
			\\&\leq C\Delta^{\bar{q}}\alpha(\Delta)^{2\bar{q}}.
		\end{split}
	\end{equation*}
	Then, by (\ref{eq33}), we have
	\begin{equation*}
		\begin{split}
			\mathbb{E}&|\bar{R}_{f}(t,\eta_{\Delta}(t),\bar{\eta}_{\Delta}(t) ,\eta_{\Delta}(t-\tau),\bar{\eta}_{\Delta}(t-\tau))|^{\bar{q}}
			\\&=C\Delta^{\bar{q}}\mathbb{E}|f_{1}(\mu(t),\bar{\eta}_{\Delta}(t),\bar{\eta}_{\Delta}(t
			-\tau))f_{\Delta}(t,\eta_{\Delta}(t),\eta_{\Delta}(t-\tau))|^{\bar{q}}
			\\&\quad+C\Delta^{\bar{q}}\mathbb{E}|f_{1}(\mu(t),\bar{\eta}_{\Delta}(t),\bar{\eta}_{\Delta}(t
			-\tau))f_{\Delta}(\mu(t),\bar{\eta}_{\Delta}(t),\bar{\eta}_{\Delta}(t
			-\tau))|^{\bar{q}}
			\\&\quad+C\mathbb{E}|f_{1}(\mu(t),\bar{\eta}_{\Delta}(t),\bar{\eta}_{\Delta}(t
			-\tau))g_{1,\Delta}(\mu(t),\bar{\eta}_{\Delta}(t),\bar{\eta}_{\Delta}(t
			-\tau))\\
			&\quad\quad \cdot g_{\Delta}(\mu(t),\bar{\eta}_{\Delta}(t),\bar{\eta}_{\Delta}(t
			-\tau))((\Delta\hat{B}(t))^{2}-\Delta)|^{\bar{q}}\\
			&\quad+C\mathbb{E}|f_{1}(\mu(t),\bar{\eta}_{\Delta}(t),\bar{\eta}_{\Delta}(t
			-\tau))g_{2,\Delta}(\mu(t),\bar{\eta}_{\Delta}(t),\bar{\eta}_{\Delta}(t
			-\tau))
			\\&\quad\quad\cdot  g_{\Delta}(\mu(t-\tau),\bar{\eta}_{\Delta}(t-\tau),\bar{\eta}_{\Delta}(t
			-2\tau))\int_{\kappa(t)}^{t}\Delta\hat{B}(s-\tau)dB(s)|^{\bar{q}}
			\\&\quad+C\Delta^{\bar{q}}\mathbb{E}|f_{2}(\mu(t),\bar{\eta}_{\Delta}(t),\bar{\eta}_{\Delta}(t
			-\tau))f_{\Delta}(t-\tau,\eta_{\Delta}(t-\tau),\eta_{\Delta}(t-2\tau))|^{\bar{q}}
			\\&\quad+C\Delta^{\bar{q}}\mathbb{E}|f_{2}(\mu(t),\bar{\eta}_{\Delta}(t),\bar{\eta}_{\Delta}(t
			-\tau))f_{\Delta}(\mu(t-\tau),\bar{\eta}_{\Delta}(t-\tau),\bar{\eta}_{\Delta}(t
			-2\tau))|^{\bar{q}}
			\\&\quad+C\mathbb{E}|f_{2}(\mu(t),\bar{\eta}_{\Delta}(t),\bar{\eta}_{\Delta}(t
			-\tau))g_{1,\Delta}(\mu(t),\bar{\eta}_{\Delta}(t),\bar{\eta}_{\Delta}(t
			-\tau))\\
			&\quad\quad\cdot  g_{\Delta}(\mu(t),\bar{\eta}_{\Delta}(t),\bar{\eta}_{\Delta}(t
			-\tau))((\Delta\hat{B}(t))^{2}-\Delta)|^{\bar{q}}
			\\&\quad+C\mathbb{E}|f_{2}(\mu(t),\bar{\eta}_{\Delta}(t),\bar{\eta}_{\Delta}(t
			-\tau))g_{2,\Delta}(\mu(t),\bar{\eta}_{\Delta}(t),\bar{\eta}_{\Delta}(t
			-\tau))
			\\&\quad\quad\cdot  g_{\Delta}(\mu(t-\tau),\bar{\eta}_{\Delta}(t-\tau),\bar{\eta}_{\Delta}(t
			-2\tau))\int_{\kappa(t)}^{t}\Delta\hat{B}(s-\tau)dB(s)|^{\bar{q}}
			\\&\quad+C\mathbb{E}|R_{f}(t,\eta_{\Delta}(t),\bar{\eta}_{\Delta}(t) ,\eta_{\Delta}(t-\tau),\bar{\eta}_{\Delta}(t-\tau))|^{\bar{q}}.
		\end{split}
	\end{equation*}
	Using H\"{o}lder's inequality and Lemma \ref{lemmaa33} gives the following estimates
	\begin{equation*}
		\mathbb{E}|(\Delta\hat{B}(t))^{2}-\Delta|^{2\bar{q}}\vee\mathbb{E}|\int_{\kappa(t)}^{t}\Delta\hat{B}(s-\tau)dB(s)|^{2\bar{q}}\leq C\Delta^{2\bar{q}},
	\end{equation*}
	\begin{equation*}
		\begin{split}
			&\mathbb{E}|f_{1}(\mu(t),\bar{\eta}_{\Delta}(t),\bar{\eta}_{\Delta}(t
			-\tau))f_{\Delta}(t,\eta_{\Delta}(t),\eta_{\Delta}(t
			-\tau))|^{\bar{q}}
			\\&\vee\mathbb{E}|f_{1}(\mu(t),\bar{\eta}_{\Delta}(t),\bar{\eta}_{\Delta}(t
			-\tau))f_{\Delta}(\mu(t),\bar{\eta}_{\Delta}(t),\bar{\eta}_{\Delta}(t
			-\tau))|^{\bar{q}}
			\\&\vee \mathbb{E}|f_{2}(\mu(t),\bar{\eta}_{\Delta}(t),\bar{\eta}_{\Delta}(t
			-\tau))f_{\Delta}(t-\tau,\eta_{\Delta}(t-\tau),\eta_{\Delta}(t
			-2\tau))|^{\bar{q}}
			\\&\vee\mathbb{E}|f_{2}(\mu(t),\bar{\eta}_{\Delta}(t),\bar{\eta}_{\Delta}(t
			-\tau))f_{\Delta}(\mu(t-\tau),\bar{\eta}_{\Delta}(t-\tau),\bar{\eta}_{\Delta}(t
			-2\tau))|^{\bar{q}}
			\\&\quad\leq C\alpha(\Delta)^{\bar{q}},
		\end{split}
	\end{equation*}
	\begin{equation*}
		\begin{split}
			&\mathbb{E}|f_{1}(\mu(t),\bar{\eta}_{\Delta}(t),\bar{\eta}_{\Delta}(t
			-\tau))g_{1,\Delta}(\mu(t),\bar{\eta}_{\Delta}(t),\bar{\eta}_{\Delta}(t
			-\tau))g_{\Delta}(\mu(t),\bar{\eta}_{\Delta}(t),\bar{\eta}_{\Delta}(t
			-\tau))|^{2\bar{q}}
			\\&\vee\mathbb{E}|f_{1}(\mu(t),\bar{\eta}_{\Delta}(t),\bar{\eta}_{\Delta}(t
			-\tau))g_{2,\Delta}(\mu(t),\bar{\eta}_{\Delta}(t),\bar{\eta}_{\Delta}(t
			-\tau))g_{\Delta}(\mu(t-\tau),\bar{\eta}_{\Delta}(t-\tau),\bar{\eta}_{\Delta}(t
			-2\tau))|^{2\bar{q}}
			\\&\vee \mathbb{E}|f_{2}(\mu(t),\bar{\eta}_{\Delta}(t),\bar{\eta}_{\Delta}(t
			-\tau))g_{1,\Delta}(\mu(t),\bar{\eta}_{\Delta}(t),\bar{\eta}_{\Delta}(t
			-\tau))g_{\Delta}(\mu(t),\bar{\eta}_{\Delta}(t),\bar{\eta}_{\Delta}(t
			-\tau))|^{2\bar{q}}
			\\&\vee\mathbb{E}|f_{2}(\mu(t),\bar{\eta}_{\Delta}(t),\bar{\eta}_{\Delta}(t
			-\tau))g_{2,\Delta}(\mu(t),\bar{\eta}_{\Delta}(t),\bar{\eta}_{\Delta}(t
			-\tau))g_{\Delta}(\mu(t-\tau),\bar{\eta}_{\Delta}(t-\tau),\bar{\eta}_{\Delta}(t
			-2\tau))|^{2\bar{q}}
			\\&\quad\leq C\alpha(\Delta)^{4\bar{q}}.
		\end{split}
	\end{equation*}
	Combining these inequalities together with H\"{o}lder's inequality gives that
	\begin{equation*}
		\mathbb{E}|\bar{R}_{f}(t,\eta_{\Delta}(t),\bar{\eta}_{\Delta}(t) ,\eta_{\Delta}(t-\tau),\bar{\eta}_{\Delta}(t-\tau))|^{\bar{q}}
		\\\leq C\Delta^{\bar{q}}\alpha(\Delta)^{2\bar{q}}.
	\end{equation*}
	Similarly, the other results could be obtained. The proof is complete.
\end{proof}

\begin{lem}\label{le35}
	Let Assumptions  \ref{a31}-\ref{a33} hold. 
	Let $L>||\xi||$ be any real number,
	and  define 
	\begin{equation*}
		\lambda_{L}=inf\{t\geq0:|x(t)|\geq L\},~~~\lambda_{\Delta,L}=inf\{t\geq0:|\eta_{\Delta}(t)|\geq L\}.
	\end{equation*}
	Then, we derive that
	\begin{equation*}
		\mathbb{P}(\lambda_{L}\leq T)\leq\frac{C}{L^{p}},~~~
		\mathbb{P}(\lambda_{\Delta,L}\leq T)\leq\frac{C}{L^{p}}.
	\end{equation*}
\end{lem}

Lemma \ref{le35} can be proved by borrowing the proof techniques in Lemma \ref{le1} and Lemma \ref{le32}.

\begin{lem}\label{le336}
	Let Assumptions  \ref{a31}-\ref{a36} hold with $p\geq2(1+\beta)q$ . 
	Let $L>||\xi||$ be any real number,
	and  suppose that there exists a sufficiently small $\Delta\in (0, \Delta^{*})$ satisfying $\Lambda^{-1}(\alpha(\Delta))\geq L$. Then we have
	\begin{equation*}
		\mathbb{E}|x(T\wedge\varrho_{\Delta,L})-\eta_{\Delta}(T\wedge\varrho_{\Delta,L})|^{2}\leq C(\Delta^{2}\alpha(\Delta)^{4}+\Delta^{2\gamma}+\Delta^{2\sigma}),
	\end{equation*}
	where $\varrho_{\Delta,L}:=\lambda_{L}\wedge\lambda_{\Delta,L}$.
\end{lem}
\begin{proof}
	Denote $\varrho_{\Delta,L}=\varrho$ for simplicity.
	Denote $\Upsilon_{\Delta}(t)=x(t)-Z_{\Delta}(t)$.
	For $0\leq s\leq t\wedge\varrho$, one can see that
	\begin{equation*}
		|x(s)|\vee|x(s-\tau)|\vee|\bar{\eta}_{\Delta}(s)|\vee|\bar{\eta}_{\Delta}(s-\tau)|\leq L\leq \Lambda^{-1}(\alpha(\Delta)).
	\end{equation*}
In addition, for $\tau\leq s\leq t\wedge\varrho$,
$|\bar{\eta}_{\Delta}(s-2\tau)|\leq L\leq \Lambda^{-1}(\alpha(\Delta)).$ Then from the previous definitions, we get that
	\begin{equation*}
		\begin{split}
			f_{\Delta}(\mu(s),\bar{\eta}_{\Delta}(s),\bar{\eta}_{\Delta}(s
			-\tau))=&f(\mu(s),\bar{\eta}_{\Delta}(s),\bar{\eta}_{\Delta}(s
			-\tau)),\\
			g_{\Delta}(\mu(s),\bar{\eta}_{\Delta}(s),\bar{\eta}_{\Delta}(s
			-\tau))=&g(\mu(s),\bar{\eta}_{\Delta}(s),\bar{\eta}_{\Delta}(s
			-\tau)),\\
			g_{1,\Delta}(\mu(s),\bar{\eta}_{\Delta}(s),\bar{\eta}_{\Delta}(s
			-\tau))=&g_{1}(\mu(s),\bar{\eta}_{\Delta}(s),\bar{\eta}_{\Delta}(s
			-\tau)),\\
			g_{2,\Delta}(\mu(s),\bar{\eta}_{\Delta}(s),\bar{\eta}_{\Delta}(s
			-\tau))=&g_{2}(\mu(s),\bar{\eta}_{\Delta}(s),\bar{\eta}_{\Delta}(s
			-\tau)),\\
			g_{\Delta}(\mu(s-\tau),\bar{\eta}_{\Delta}(s-\tau),\bar{\eta}_{\Delta}(s
			-2\tau))=&g(\mu(s-\tau),\bar{\eta}_{\Delta}(s-\tau),\bar{\eta}_{\Delta}(s
			-2\tau)).
		\end{split}
	\end{equation*}
	Applying It\^{o}'s formula and (\ref{eqeq07}) yields that
	\begin{equation*}
		\begin{split}
			&\mathbb{E}|\Upsilon_{\Delta}(t\wedge\varrho)|^{2}\\
			\leq&\theta^{2}|f_{\Delta}(0,\xi(0),\xi(-\tau))|^{2}\Delta^{2}
			\\&+\mathbb{E}\int_{0}^{t\wedge\varrho}2[\Upsilon_{\Delta}^{T}(s)\left(f(s,x(s),x(s
			-\tau))-f_{\Delta}(\mu(s),\bar{\eta}_{\Delta}(s),\bar{\eta}_{\Delta}(s
			-\tau))\right)
			\\&+\frac{1}{2}|g(s,x(s),x(s
			-\tau))-g_{\Delta}(\mu(s),\bar{\eta}_{\Delta}(s),\bar{\eta}_{\Delta}(s
			-\tau))
			\\&-g_{1,\Delta}(\mu(s),\bar{\eta}_{\Delta}(s),\bar{\eta}_{\Delta}(s
			-\tau))g_{\Delta}(\mu(s),\bar{\eta}_{\Delta}(s),\bar{\eta}_{\Delta}(s
			-\tau))\Delta \hat{B}(s)
			\\&-g_{2,\Delta}(\mu(s),\bar{\eta}_{\Delta}(s),\bar{\eta}_{\Delta}(s
			-\tau))g_{\Delta}(\mu(s-\tau),\bar{\eta}_{\Delta}(s-\tau),\bar{\eta}_{\Delta}(s
			-2\tau))\Delta \hat{B}(s-\tau)|^{2}]ds
		\end{split}
	\end{equation*}
			\begin{equation*}
				\begin{split}
			\leq& C\Delta^{2}\alpha(\Delta)^{2}\\
			&+C\mathbb{E}\int_{0}^{t\wedge\varrho}\left[( x(s)-\eta_{\Delta}(s))^T(f(s,x(s),x(s
			-\tau))-f(\mu(s),\bar{\eta}_{\Delta}(s),\bar{\eta}_{\Delta}(s
			-\tau)))\right.
			\\&+\left.|g(s,x(s),x(s
			-\tau))-g(\mu(s),\eta_{\Delta}(s),\eta_{\Delta}(s
			-\tau))|^{2}\right]ds
			\\&+C\mathbb{E}\int_{0}^{t\wedge\varrho}|\bar{R}_{g}(s,\eta_{\Delta}(s),\bar{\eta}_{\Delta}(s) ,\eta_{\Delta}(s-\tau),\bar{\eta}_{\Delta}(s-\tau))|^{2}ds
			\\&+C\mathbb{E}\int_{0}^{t\wedge\varrho}\theta\Delta f_{\Delta}^T(s,\eta_{\Delta}(s),\eta_{\Delta}(s-\tau))\\
			&\quad\cdot (f(s,x(s),x(s
			-\tau))-f(\mu(s),\bar{\eta}_{\Delta}(s),\bar{\eta}_{\Delta}(s
			-\tau))) ds
			\\=:& C\Delta^{2}\alpha(\Delta)^{2}+J_{1}+J_{2}+J_{3}.
		\end{split}
	\end{equation*}
	Using Young's inequality gives that
	\begin{equation}
		\begin{split}
			J_{1}\leq&C\mathbb{E}\int_{0}^{t\wedge\varrho}\left[( x(s)-\eta_{\Delta}(s))^T(f(s,x(s),x(s
			-\tau))-f(s,\eta_{\Delta}(s),\eta_{\Delta}(s
			-\tau)))\right.
			\\&\left.+(q-1)|g(s,x(s),x(s
			-\tau))-g(s,\eta_{\Delta}(s),\eta_{\Delta}(s
			-\tau))|^{2}\right]ds
			\\& +C\mathbb{E}\int_{0}^{t\wedge\varrho}( x(s)-\eta_{\Delta}(s))^T(f(s,\eta_{\Delta}(s),\eta_{\Delta}(s
			-\tau))-f(\mu(s),\eta_{\Delta}(s),\eta_{\Delta}(s
			-\tau))) ds
			\\& +C\mathbb{E}\int_{0}^{t\wedge\varrho}( x(s)-\eta_{\Delta}(s))^T(f(\mu(s),\eta_{\Delta}(s),\eta_{\Delta}(s
			-\tau))-f(\mu(s),\bar{\eta}_{\Delta}(s),\bar{\eta}_{\Delta}(s
			-\tau))) ds
			\\&+C\mathbb{E}\int_{0}^{t\wedge\varrho}\frac{q-1}{q-2}|g(s,\eta_{\Delta}(s),\eta_{\Delta}(s
			-\tau))-g(\mu(s),\eta_{\Delta}(s),\eta_{\Delta}(s
			-\tau))|^{2}ds
			\\=:&J_{11}+J_{12}+J_{13}+J_{14}.
		\end{split}
	\end{equation}
	By Assumption \ref{a32}, we have
	\begin{equation*}
		\begin{split}
			J_{11}\leq&C\mathbb{E}\int_{0}^{t\wedge\varrho}(|x(s)-\eta_{\Delta}(s)|^{2}+|x(s-\tau)-\eta_{\Delta}(s-\tau)|^{2}\\
			&-U(x(s),\eta_{\Delta}(s))
			+U(x(s-\tau),\eta_{\Delta}(s-\tau)))ds
			\\\leq&C\mathbb{E}\int_{0}^{t\wedge\varrho}|x(s)-\eta_{\Delta}(s)|^{2}+C\mathbb{E}\int_{-\tau}^{0}|\xi(s)-\xi(\kappa(s))|^{2}ds
			\\&+C\mathbb{E}\int_{0}^{t\wedge\varrho}\left(-U(x(s),\eta_{\Delta}(s))
			+U(x(s-\tau),\eta_{\Delta}(s-\tau))\right)ds
			\\\leq&C\mathbb{E}\int_{0}^{t\wedge\varrho}|x(s)-\eta_{\Delta}(s)|^{2}ds+C \Delta^{2\gamma}+C\int_{-\tau}^{0}U(\xi(s),\xi(\kappa(s)))ds
			\\\leq&C\mathbb{E}\int_{0}^{t\wedge\varrho}|x(s)-\eta_{\Delta}(s)|^{2}ds+C \Delta^{2\gamma}+C\int_{-\tau}^{0}\rho_{\varsigma}|\xi(s)-\xi(\kappa(s))|^{2}ds
			\\\leq&C\int_{0}^{t}\mathbb{E}|x(s\wedge\varrho)-\eta_{\Delta}(s\wedge\varrho)|^{2}ds+C \Delta^{2\gamma}.
		\end{split}
	\end{equation*}
By Assumption \ref{a34}, one can see that
	\begin{equation*}
		\begin{split}
			J_{12}\leq&C\int_{0}^{t}\mathbb{E}|x(s\wedge\varrho)-\eta_{\Delta}(s\wedge\varrho)|^{2}ds\\
			&+C\mathbb{E}\int_{0}^{t\wedge\varrho}(1+|\eta_{\Delta}(s)|^{2(\beta+1)}+|\eta_{\Delta}(s-\tau)|^{2(\beta+1)})|s-\kappa(s)|^{2\sigma}ds
			\\\leq&C\int_{0}^{t}\mathbb{E}|x(s\wedge\varrho)-\eta_{\Delta}(s\wedge\varrho)|^{2}ds
			+C\Delta^{2\sigma}.
		\end{split}
	\end{equation*}
	Similar to \cite{42}, we have
	\begin{equation*}
		\begin{split}
			J_{13}\leq&C\mathbb{E}\int_{0}^{t\wedge\varrho} (x(s)-\eta_{\Delta}(s))^T\\
			&\quad\cdot [f_{1}(\mu(s),\bar{\eta}_{\Delta}(s),\bar{\eta}_{\Delta}(s
			-\tau))\int_{\mu(s)}^{s}g_{\Delta}(\mu(r),\bar{\eta}_{\Delta}(r),\bar{\eta}_{\Delta}(r
			-\tau))dB(r)
			\\&+f_{2}(\mu(s),\bar{\eta}_{\Delta}(s),\bar{\eta}_{\Delta}(s
			-\tau))\int_{\mu(s)-\tau}^{s-\tau}g_{\Delta}(\mu(r),\bar{\eta}_{\Delta}(r),\bar{\eta}_{\Delta}(r
			-\tau))dB(r)
			\\&+\bar{R}_{f}(s,\eta_{\Delta}(s),\bar{\eta}_{\Delta}(s) ,\eta_{\Delta}(s-\tau),\bar{\eta}_{\Delta}(s-\tau))] ds
			\\\leq&C\mathbb{E}\int_{0}^{t\wedge\varrho}\left(| x(s)-\eta_{\Delta}(s)|^{2}+|\bar{R}_{f}(s,\eta_{\Delta}(s),\bar{\eta}_{\Delta}(s) ,\eta_{\Delta}(s-\tau),\bar{\eta}_{\Delta}(s-\tau))|^{2}\right)ds\\
			&+C\Delta^{2}\alpha(\Delta)^{4}+C\Delta^{2\gamma}
			\\\leq&C\int_{0}^{t}\mathbb{E}|x(s\wedge\varrho)-\eta_{\Delta}(s\wedge\varrho)|^{2}ds+C\Delta^{2}\alpha(\Delta)^{4}+C\Delta^{2\gamma}.
		\end{split}
	\end{equation*}
	Borrowing the technique in the estimation of $J_{12}$ gives that
	\begin{equation*}
		J_{14}\leq C\Delta^{2\sigma}.
	\end{equation*}
	Thus,
	\begin{equation}\label{jjj1}
		J_{1}\leq C\int_{0}^{t}\mathbb{E}|x(s\wedge\varrho)-\eta_{\Delta}(s\wedge\varrho)|^{2}ds+C\Delta^{2\gamma}+C\Delta^{2\sigma}+C\Delta^{2}\alpha(\Delta)^{4}.
	\end{equation}
	By Lemma \ref{le34}, we derive that
	\begin{equation}\label{jjj2}
		J_{2}\leq C\Delta^{2}\alpha(\Delta)^{4}.
	\end{equation}
	Similar to $J_{1}$, using Young's inequality gives that
	\begin{equation}\label{jjj3}
		J_{3}\leq C\int_{0}^{t}\mathbb{E}|x(s\wedge\varrho)-\eta_{\Delta}(s\wedge\varrho)|^{2}ds+C\Delta^{2\gamma}+C\Delta^{2\sigma}+C\Delta^{2}\alpha(\Delta)^{4}.
	\end{equation}
	Combining (\ref{jjj1}) - (\ref{jjj3}) together yields that
	\begin{equation*}
		\mathbb{E}|\Upsilon_{\Delta}(t\wedge\varrho)|^{2}\leq C\int_{0}^{t}\mathbb{E}|x(s\wedge\varrho)-\eta_{\Delta}(s\wedge\varrho)|^{2}ds+C\Delta^{2\gamma}+C\Delta^{2\sigma}+C\Delta^{2}\alpha(\Delta)^{4}.
	\end{equation*}
	Thus,
	\begin{equation*}
		\begin{split}
			&\mathbb{E}|x(t\wedge\varrho)-\eta_{\Delta}(t\wedge\varrho)|^{2}\\ \leq&C\left(\Delta^{2}\alpha(\Delta)^{2}+\mathbb{E}|\Upsilon_{\Delta}(t\wedge\varrho)|^{2}\right)
			\\\leq & C\left(\int_{0}^{t}\mathbb{E}|x(s\wedge\varrho)-\eta_{\Delta}(s\wedge\varrho)|^{2}ds+\Delta^{2\gamma}+\Delta^{2\sigma}+\Delta^{2}\alpha(\Delta)^{4}\right).
		\end{split}
	\end{equation*}
	Thanks to the Gronwall inequality, the  result follows. The proof is complete.
\end{proof}

\begin{thm}\label{theo37}
	Let Assumptions  \ref{a31}-\ref{a36} hold with $p\geq2(1+\beta)q$. For sufficiently small $\Delta\in (0, \Delta^{*})$, assume that 
	\begin{equation}\label{ttc1}
		\alpha(\Delta)\geq\Lambda\left((\Delta^{2}\alpha(\Delta)^{4}\vee\Delta^{2(\gamma\wedge\sigma)})^{\frac{-1}{p-2}}\right).
	\end{equation}
	Then we derive that, for any such small $\Delta$,
	\begin{equation}\label{ress1}
		\mathbb{E}|x(T)-\eta_{\Delta}(T)|^{2}\leq C(\Delta^{2}\alpha(\Delta)^{4}\vee\Delta^{2(\gamma\wedge\sigma)}),
	\end{equation}
	and
	\begin{equation}\label{ress2}
		\mathbb{E}|x(T)-\bar{\eta}_{\Delta}(T)|^{2}\leq C(\Delta^{2}\alpha(\Delta)^{4}\vee\Delta^{2(\gamma\wedge\sigma)}).
	\end{equation}
\end{thm}
\begin{proof} 
	Let $\hat{\Upsilon}_{\Delta}(t)=x(t)-\eta_{\Delta}(t)$ for $t\in[0,T]$ and $\Delta\in (0, \Delta^{*})$. Denote $\varrho=\varrho_{\Delta,L}$ for simplicity.
	One can see that
	\begin{equation*}
		\mathbb{E}|\hat{\Upsilon}_{\Delta}(T)|^{2}= \mathbb{E}\left(|\hat{\Upsilon}_{\Delta}(T)|^{2}\mathbb{I}_{\{\varrho>T\}}\right)+\mathbb{E}\left(|\hat{\Upsilon}_{\Delta}(T)|^{2}\mathbb{I}_{\{\varrho\leq T\}}\right).
	\end{equation*}
	Let $\Gamma>0$ be arbitrary. By Young's inequality, we derive that
	\begin{equation*}
		\begin{split}
			a^{2}b=(\Gamma a^{p})^{\frac{2}{p}}(\frac{b^{p/(p-2)}}{\Gamma^{2/(p-2)}})^{\frac{p-2}{p}}
			\leq\frac{2\Gamma}{p}a^{p}+\frac{p-2}{p\Gamma^{2/(p-2)}}b^{p/(p-2)}, \forall a,b>0.
		\end{split}
	\end{equation*}
	Therefore, 
	\begin{equation*}
		\mathbb{E}\left(|\hat{\Upsilon}_{\Delta}(T)|^{2}\mathbb{I}_{\{\varrho\leq T\}}\right)\leq\frac{2\Gamma}{p}\mathbb{E}|\hat{\Upsilon}_{\Delta}(T)|^{p}+\frac{p-2}{p\Gamma^{2/(p-2)}}\mathbb{P}(\varrho\leq T).
	\end{equation*}
	Using Lemma \ref{le1} and Lemma \ref{le32} gives that
	\begin{equation*}
		\mathbb{E}|\hat{\Upsilon}_{\Delta}(T)|^{p}\leq C (\mathbb{E}|x(T)|^{p}+ \mathbb{E}|\eta_{\Delta}(T)|^{p})\leq C.
	\end{equation*}
	Applying Lemma \ref{le35} yields that
	\begin{equation*}
		\mathbb{P}(\varrho\leq T)\leq \mathbb{P}(\lambda_{L}\leq T)+\mathbb{P}(\lambda_{\Delta,L}\leq T)\leq\frac{C}{L^{p}}.
	\end{equation*}
	Choose $\Gamma=\Delta^{2}\alpha(\Delta)^{4}\vee\Delta^{2(\gamma\wedge\sigma)}$ and $L=\left(\Delta^{2}\alpha(\Delta)^{4}\vee\Delta^{2(\gamma\wedge\sigma)}\right)^{\frac{-1}{p-2}}$. Then  we have
	\begin{equation*}
		\mathbb{E}|\hat{\Upsilon}_{\Delta}(T)|^{2}\leq \mathbb{E}|\hat{\Upsilon}_{\Delta}(T\wedge\varrho)|^{2}+ C(\Delta^{2}\alpha(\Delta)^{4}\vee\Delta^{2(\gamma\wedge\sigma)}).
	\end{equation*}
	Using the condition (\ref{ttc1})  gives that
	\begin{equation*}
		\Lambda^{-1}(\alpha(\Delta))\geq(\Delta^{2}\alpha(\Delta)^{4}\vee\Delta^{2(\gamma\wedge\sigma)})^{\frac{-1}{p-2}}=L.
	\end{equation*}
So, we get from Lemma \ref{le336} that
\begin{equation*}
	\mathbb{E}|\hat{\Upsilon}_{\Delta}(T)|^{2}\leq C(\Delta^{2}\alpha(\Delta)^{4}\vee\Delta^{2(\gamma\wedge\sigma)}).
\end{equation*}
Therefore, the desired result (\ref{ress1}) is obtained. Then combining Lemma \ref{le31} and (\ref{ress1}) gives (\ref{ress2}).
The proof is complete.
\end{proof}

\begin{cor}
	Under the assumptions in Theorem \ref{theo37}. Let $\Lambda(w)=c^{*}w^{\beta+1}$ for $w\geq1$, $c^{*}\geq0$ and $\alpha(\Delta)=\Delta^{-\frac{\varepsilon}{2}}$ for some $\varepsilon\in (0,\frac{1}{2} )$, $\Delta\in (0, \Delta^{*})$. 
	Then we derive that
	\begin{equation*}
		\mathbb{E}|x(T)-\eta_{\Delta}(T)|^{2}\leq C\Delta^{2[(1-\varepsilon)\wedge\gamma\wedge\sigma]},
	\end{equation*}
	and
	\begin{equation*}
		\mathbb{E}|x(T)-\bar{\eta}_{\Delta}(T)|^{2}\leq C\Delta^{2[(1-\varepsilon)\wedge\gamma\wedge\sigma]}.
	\end{equation*}
\end{cor}

\begin{rem}\label{remar2}
	The better result can be given if we impose the stronger conditions on $\gamma$ and $\sigma$. For instance, suppose that $\gamma,\sigma\in [1-\varepsilon,1]$. Then we have $(1-\varepsilon)$-order convergence rate in $\mathcal{L}^{2}$ sense.
\end{rem}
To get the convergence rate in  $\mathcal{L}^{\bar{q}}$($\bar{q}>2$) sense, 
Assumption \ref{a32} needs to be replaced by Assumption \ref{a390}.
\begin{ass}\label{a390}
	The exist constants $\hat{K}_{1}>0$ and $\hat{q}\in (2,q)$ such that
	\begin{equation*}
		(a-\bar{a})^T(f(t,a,b)-f(t,\bar{a},\bar{b}))+(\hat{q}-1)\left|g(t,a,b)-g(t,\bar{a},\bar{b})\right|^{2}\leq \hat{K}_{1}(|a-\bar{a}|^{2}+|b-\bar{b}|^{2})
	\end{equation*}
	for any $t\in [0,T]$ and $a,\bar{a},b,\bar{b}\in \mathbb{R}$.
\end{ass}

Since many techniques used in Lemma \ref{le336} are applied here, we mainly state different proof processes in the following lemma, but omit the similar processes.

\begin{lem}\label{le310}
	Let Assumptions \ref{a31}, \ref{a33}-\ref{a36} and \ref{a390} hold with $p\geq2(1+\beta)\hat{q}$.
		Let $L>||\xi||$ be any real number,
	and  suppose that there exists a sufficiently small $\Delta\in (0, \Delta^{*})$ satisfying $\Lambda^{-1}(\alpha(\Delta))\geq L$.
	 Then we have, for any $\bar{q}\in (2,\hat{q})$, 
	\begin{equation*}
		\mathbb{E}|x(T\wedge\varrho_{\Delta,L})-\eta_{\Delta}(T\wedge\varrho_{\Delta,L})|^{\bar{q}}\leq  C(\Delta^{\bar{q}}\alpha(\Delta)^{2\bar{q}}\vee\Delta^{\bar{q}(\gamma\wedge\sigma)}),
	\end{equation*}
	where $\varrho_{\Delta,L}:=\lambda_{L}\wedge\lambda_{\Delta,L}$. 
\end{lem}
\begin{proof}
	For simplicity, let $\varrho_{\Delta,L}=\varrho$ and  $\Upsilon_{\Delta}(t)=x(t)-Z_{\Delta}(t)$. We derive from It\^{o}'s formula that
	\begin{equation*}
		\begin{split}
			&\mathbb{E}|\Upsilon_{\Delta}(t\wedge\varrho)|^{\bar{q}}\\ \leq&\theta^{\bar{q}}|f_{\Delta}(0,\xi(0),\xi(-\tau))|^{\bar{q}}\Delta^{\bar{q}}
			\\ &+\mathbb{E}\int_{0}^{t\wedge\varrho}\bar{q}|\Upsilon_{\Delta}(s)|^{\bar{q}-2}[ \Upsilon_{\Delta}^T(s)(f(s,x(s),x(s-\tau))-f_{\Delta}(\mu(s),\bar{y}_\Delta(s),\bar{y}_\Delta(s-\tau)))			\\ &+\frac{\bar{q}-1}{2}\left|g(s,x(s),x(s-\tau))-g_{\Delta}(\mu(s),\bar{y}_\Delta(s),\bar{y}_\Delta(s-\tau))\right.
			\\ &-g_{1,\Delta}(\mu(s),\bar{y}_\Delta(s),\bar{y}_\Delta(s-\tau))g_{\Delta}(\mu(s),\bar{y}_\Delta(s),\bar{y}_\Delta(s-\tau))\Delta\hat{B}(s)
			\\ &\left.-g_{2,\Delta}(\mu(s),\bar{y}_\Delta(s),\bar{y}_\Delta(s-\tau))g_{\Delta}(\mu(s-\tau),\bar{y}_\Delta(s-\tau),\bar{y}_\Delta(s-2\tau))\Delta\hat{B}(s-\tau)\right|^{2}]ds
			\\\leq&C\Delta^{\bar{q}}\alpha(\Delta)^{\bar{q}}+C\mathbb{E}\int_{0}^{t\wedge\varrho}|\Upsilon_{\Delta}(s)|^{\bar{q}-2}\left[( x(s)-\eta_{\Delta}(s))^T\right.\\
			&\quad\cdot (f(s,x(s),x(s-\tau))-f(\mu(s),\bar{y}_\Delta(s),\bar{y}_\Delta(s-\tau)))
			\\ &\left.+(\bar{q}-1)|g(s,x(s),x(s-\tau))-g(\mu(s),y_\Delta(s),y_\Delta(s-\tau))|^{2}\right]ds
			\\ &+C\mathbb{E}\int_{0}^{t\wedge\varrho}|\Upsilon_{\Delta}(s)|^{\bar{q}-2}|\bar{R}_{g}(s,\eta_{\Delta}(s),\bar{\eta}_{\Delta}(s) ,\eta_{\Delta}(s-\tau),\bar{\eta}_{\Delta}(s-\tau))|^{2}ds
			\\ &+C\mathbb{E}\int_{0}^{t\wedge\varrho}|\Upsilon_{\Delta}(s)|^{\bar{q}-2}\theta\Delta f_{\Delta}^T(s,\eta_{\Delta}(s),\eta_{\Delta}(s-\tau))\\
			&\quad\cdot  (f(s,x(s),x(s-\tau))-f(\mu(s),\bar{y}_\Delta(s),\bar{y}_\Delta(s-\tau))) ds
			\\=&:C\Delta^{\bar{q}}\alpha(\Delta)^{\bar{q}}+B_{1}+B_{2}+B_{3}.
		\end{split}
	\end{equation*}
	 For $\bar{q}\in (2,\hat{q})$, note that
	\begin{equation}\label{eqq5}
		\begin{split}
			&(\bar{q}-1)|g(s,x(s),x(s-\tau))-g(\mu(s),\eta_{\Delta}(s),\eta_{\Delta}((s-\tau))|^{2}
			\\&\leq(\hat{q}-1)|g(s,x(s),x(s-\tau))-g(s,\eta_{\Delta}(s),\eta_{\Delta}((s-\tau))|^{2}
			\\&+\frac{(\bar{q}-1)(\hat{q}-1)}{\hat{q}
				-\bar{q}}|g(s,\eta_{\Delta}(s),\eta_{\Delta}(s-\tau))-g(\mu(s),\eta_{\Delta}(s),\eta_{\Delta}(s-\tau))|^{2}.
		\end{split}
	\end{equation}
	From (\ref{eqq5}), one can see that
	\begin{equation*}
		\begin{split}
			B_{1}\leq&C\mathbb{E}\int_{0}^{t\wedge\varrho}|\Upsilon_{\Delta}(s)|^{\bar{q}-2}\left[( x(s)-\eta_{\Delta}(s))^T\right.
			\\
			&\quad\cdot (f(s,x(s),x(s-\tau))-f(s,\eta_{\Delta}(s),\eta_{\Delta}(s-\tau)))
			\\&\quad\left.+(\hat{q}-1)|g(s,x(s),x(s-\tau))-g(s,\eta_{\Delta}(s),\eta_{\Delta}(s-\tau))|^{2}\right]ds
			\\ &+C\mathbb{E}\int_{0}^{t\wedge\varrho}|\Upsilon_{\Delta}(s)|^{\bar{q}-2}( x(s)-\eta_{\Delta}(s))^T\\
			&\quad\cdot (f(s,\eta_{\Delta}(s),\eta_{\Delta}(s-\tau))-f(\mu(s),\eta_{\Delta}(s),\eta_{\Delta}(s-\tau))) ds
			\\ &+C\mathbb{E}\int_{0}^{t\wedge\varrho}|\Upsilon_{\Delta}(s)|^{\bar{q}-2}( x(s)-\eta_{\Delta}(s))^T\\
			&\quad\cdot (f(\mu(s),\eta_{\Delta}(s),\eta_{\Delta}(s-\tau))-f(\mu(s),\bar{\eta}_{\Delta}(s),\bar{\eta}_{\Delta}(s-\tau))) ds
			\end{split}
	\end{equation*}
		\begin{equation*}
			\begin{split}
			\\ &+C\mathbb{E}\int_{0}^{t\wedge\varrho}\frac{(\bar{q}-1)(\hat{q}-1)}{\hat{q}
				-\bar{q}}|\Upsilon_{\Delta}(s)|^{\bar{q}-2}\\
			&\quad\cdot |g(s,\eta_{\Delta}(s),\eta_{\Delta}(s-\tau))-g(\mu(s),\eta_{\Delta}(s),\eta_{\Delta}(s-\tau))|^{2}ds
			\\=:&B_{11}+B_{12}+B_{13}+B_{14}.
		\end{split}
	\end{equation*}
	By Assumption \ref{a390}, we get that
	\begin{equation*}
		\begin{split}
			B_{11}\leq&C\mathbb{E}\int_{0}^{t\wedge\varrho}|\Upsilon_{\Delta}(s)|^{\bar{q}-2}\left(|x(s)-\eta_{\Delta}(s)|^{2}+|x(s-\tau)-\eta_{\Delta}(s-\tau)|^{2}\right)ds
			\\\leq&C\mathbb{E}\int_{0}^{t\wedge\varrho}|\Upsilon_{\Delta}(s)|^{\bar{q}}ds+C\mathbb{E}\int_{0}^{t\wedge\varrho}|x(s)-\eta_{\Delta}(s)|^{\bar{q}}ds
			+C\mathbb{E}\int_{-\tau}^{0}|\xi(s)-\xi(\mu(s))|^{\bar{q}}ds
			\\\leq&C\int_{0}^{t}\mathbb{E}|\Upsilon_{\Delta}(s\wedge\varrho)|^{\bar{q}}ds+C\int_{0}^{t}\mathbb{E}|x(s\wedge\varrho)-\eta_{\Delta}(s\wedge\varrho)|^{\bar{q}}ds
			+C\Delta^{\bar{q}\gamma}.
		\end{split}
	\end{equation*}
	Applying the similar technique in Lemma \ref{le336} means that
	\begin{equation*}
		B_{12}\leq C\int_{0}^{t}\mathbb{E}|\Upsilon_{\Delta}(s\wedge\varrho)|^{\bar{q}}ds+C\int_{0}^{t}\mathbb{E}|x(s\wedge\varrho)-\eta_{\Delta}(s\wedge\varrho)|^{\bar{q}}ds
		+C\Delta^{\bar{q}\sigma},
	\end{equation*}
	\begin{equation*}
			\begin{split}
		B_{13}\leq &C\int_{0}^{t}\mathbb{E}|\Upsilon_{\Delta}(s\wedge\varrho)|^{\bar{q}}ds+C\int_{0}^{t}\mathbb{E}|x(s\wedge\varrho)-\eta_{\Delta}(s\wedge\varrho)|^{\bar{q}}ds\\
		&+C\Delta^{\bar{q}}\alpha(\Delta)^{2\bar{q}}+C\Delta^{\bar{q}\gamma},
			\end{split}
	\end{equation*}
	\begin{equation*}
		B_{14}\leq C\int_{0}^{t}\mathbb{E}|\Upsilon_{\Delta}(s\wedge\varrho)|^{\bar{q}}ds+C\Delta^{\bar{q}\sigma}.
	\end{equation*}
	Hence,
	\begin{equation}\label{bbb1}
		\begin{split}
		B_{1}\leq &C\int_{0}^{t}\mathbb{E}|\Upsilon_{\Delta}(s\wedge\varrho)|^{\bar{q}}ds+C\int_{0}^{t}\mathbb{E}|x(s\wedge\varrho)-\eta_{\Delta}(s\wedge\varrho)|^{\bar{q}}ds\\
		&+C\Delta^{\bar{q}}\alpha(\Delta)^{2\bar{q}}+C\Delta^{\bar{q}\gamma}+C\Delta^{\bar{q}\sigma}.
		\end{split}
	\end{equation}
	We derive from Lemma \ref{le34} that
	\begin{equation}\label{bbb2}
		B_{2}\leq C\int_{0}^{t}\mathbb{E}|\Upsilon_{\Delta}(s\wedge\varrho)|^{\bar{q}}ds
		+C\Delta^{\bar{q}}\alpha(\Delta)^{2\bar{q}}.
	\end{equation}
	Moreover, applying the technique in the estimation of $B_{1}$  yields that
	\begin{equation}\label{bbb3}
	\begin{split}
		B_{3}\leq &C\int_{0}^{t}\mathbb{E}|\Upsilon_{\Delta}(s\wedge\varrho)|^{\bar{q}}ds+C\int_{0}^{t}\mathbb{E}|x(s\wedge\varrho)-\eta_{\Delta}(s\wedge\varrho)|^{\bar{q}}ds\\
		&+C\Delta^{\bar{q}}\alpha(\Delta)^{2\bar{q}}+C\Delta^{\bar{q}\gamma}+C\Delta^{\bar{q}\sigma}.
	\end{split}
	\end{equation}
	Combining (\ref{bbb1}) - (\ref{bbb3}) can give that
	\begin{equation*}
	\begin{split}
		\mathbb{E}|\Upsilon_{\Delta}(t\wedge\varrho)|^{\bar{q}}\leq &C\int_{0}^{t}\mathbb{E}|\Upsilon_{\Delta}(s\wedge\varrho)|^{\bar{q}}ds+C\int_{0}^{t}\mathbb{E}|x(s\wedge\varrho)-\eta_{\Delta}(s\wedge\varrho)|^{\bar{q}}ds\\
		&+C\Delta^{\bar{q}}\alpha(\Delta)^{2\bar{q}}+C\Delta^{\bar{q}\gamma}+C\Delta^{\bar{q}\sigma}.
	\end{split}
	\end{equation*}
	One can get from Gronwall's inequality that
	\begin{equation*}
	\begin{split}
		\mathbb{E}|\Upsilon_{\Delta}(t\wedge\varrho)|^{\bar{q}}\leq &C\int_{0}^{t}\mathbb{E}|x(s\wedge\varrho)-\eta_{\Delta}(s\wedge\varrho)|^{\bar{q}}ds+C\Delta^{\bar{q}}\alpha(\Delta)^{2\bar{q}}+C\Delta^{\bar{q}\gamma}+C\Delta^{\bar{q}\sigma}.
	\end{split}
	\end{equation*}
	Then we derive that
	\begin{equation*}
		\begin{split}
			&\mathbb{E}|x(t\wedge\varrho)-\eta_{\Delta}(t\wedge\varrho)|^{\bar{q}}\\\leq&C\Delta^{\bar{q}}\alpha(\Delta)^{\bar{q}}
			+C\mathbb{E}|\Upsilon_{\Delta}(t\wedge\varrho)|^{\bar{q}}
			\\\leq& C\int_{0}^{t}\mathbb{E}|x(s\wedge\varrho)-\eta_{\Delta}(s\wedge\varrho)|^{\bar{q}}ds
			+C\Delta^{\bar{q}}\alpha(\Delta)^{2\bar{q}}+C\Delta^{\bar{q}\gamma}+C\Delta^{\bar{q}\sigma}.
		\end{split}
	\end{equation*}
	The desired result can be obtained due to the Gronwall inequality. We complete the proof.
\end{proof}

\begin{thm}\label{theo311}
	Let Assumptions \ref{a31}, \ref{a33}-\ref{a36} and \ref{a390} hold with $p\geq 2 (\beta+1)\hat{q}$.  For any sufficiently small $\Delta\in (0, \Delta^{*})$,  assume that
	\begin{equation}\label{eqq10}
		\alpha(\Delta)\geq\Lambda\left((\Delta^{\bar{q}}\alpha(\Delta)^{2\bar{q}}\vee\Delta^{\bar{q}(\gamma\wedge\sigma)})^{\frac{-1}{p-\bar{q}}}\right).
	\end{equation}
	Then, for $\bar{q}\in (2,\hat{q})$ and such small $\Delta$, we have
	\begin{equation}\label{lare1}
		\mathbb{E}|x(T)-\eta_{\Delta}(T)|^{\bar{q}}\leq C(\Delta^{\bar{q}}\alpha(\Delta)^{2\bar{q}}\vee\Delta^{\bar{q}(\gamma\wedge\sigma)}),
	\end{equation}
	and
	\begin{equation}\label{lare2}
		\mathbb{E}|x(T)-\bar{\eta}_{\Delta}(T)|^{\bar{q}}\leq C(\Delta^{\bar{q}}\alpha(\Delta)^{2\bar{q}}\vee\Delta^{\bar{q}(\gamma\wedge\sigma)}).
	\end{equation}
\end{thm}
\begin{proof}
	Let $\hat{\Upsilon}_{\Delta}(t)=x(t)-\eta_{\Delta}(t)$ and $\varrho=\varrho_{\Delta,L}$ for simplicity. Obviously,
	\begin{equation*}
		\mathbb{E}|\hat{\Upsilon}_{\Delta}(T)|^{\bar{q}}= \mathbb{E}\left(|\hat{\Upsilon}_{\Delta}(T)|^{\bar{q}}\mathbb{I}_{\{\varrho>T\}}\right)+\mathbb{E}\left(|\hat{\Upsilon}_{\Delta}(T)|^{\bar{q}}\mathbb{I}_{\{\varrho\leq T\}}\right).
	\end{equation*}
	Let $\hat{\Gamma}>0$ be arbitrary. Applying Young's inequality gives that 
	\begin{equation*}
		\begin{split}
			a^{\bar{q}}b=&(\hat{\Gamma} a^{p})^{\frac{\bar{q}}{p}}(\frac{b^{p/(p-\bar{q})}}{\hat{\Gamma}^{\bar{q}/(p-\bar{q})}})^{\frac{p-\bar{q}}{p}}
		\leq\frac{\bar{q}\hat{\Gamma}}{p}a^{p}+\frac{p-\bar{q}}{p\hat{\Gamma}^{\bar{q}/(p-\bar{q})}}b^{p/(p-\bar{q})},~~~ \forall a,b>0.
		\end{split}
	\end{equation*}
	So,
	\begin{equation*}
		\mathbb{E}\left(|\hat{\Upsilon}_{\Delta}(T)|^{\bar{q}}\mathbb{I}_{\{\varrho\leq T\}}\right)\leq\frac{\bar{q}\hat{\Gamma}}{p}\mathbb{E}|\hat{\Upsilon}_{\Delta}(T)|^{p}+\frac{p-\bar{q}}{p\hat{\Gamma}^{\bar{q}/(p-\bar{q})}}\mathbb{P}(\varrho\leq T).
	\end{equation*}
	Note that
	$\mathbb{E}|\hat{\Upsilon}_{\Delta}(T)|^{p}\leq C$
	and
$\mathbb{P}(\varrho\leq T)\leq \frac{C}{L^{p}}.$
	Then set $\hat{\Gamma}=\Delta^{\bar{q}}\alpha(\Delta)^{2\bar{q}}\vee\Delta^{\bar{q}(\gamma\wedge\sigma)}$, $L=(\Delta^{\bar{q}}\alpha(\Delta)^{2\bar{q}}\vee\Delta^{\bar{q}(\gamma\wedge\sigma)})^{\frac{-1}{p-\bar{q}}}$.
	Thus, we get that 
	\begin{equation*}
		\mathbb{E}|\hat{\Upsilon}_{\Delta}(T)|^{\bar{q}}\leq \mathbb{E}|\hat{\Upsilon}_{\Delta}(T\wedge\varrho)|^{\bar{q}}+C(\Delta^{\bar{q}}\alpha(\Delta)^{2\bar{q}}\vee\Delta^{\bar{q}(\gamma\wedge\sigma)}).
	\end{equation*}
	The condition (\ref{eqq10}) means that
	\begin{equation*}
		\Lambda^{-1}(\alpha(\Delta))\geq(\Delta^{\bar{q}}\alpha(\Delta)^{2\bar{q}}\vee\Delta^{\bar{q}(\gamma\wedge\sigma)})^{\frac{-1}{p-\bar{q}}}=L.
	\end{equation*}
	Due to Lemma \ref{le310}, we have
	\begin{equation*}
		\mathbb{E}|\hat{\Upsilon}_{\Delta}(T)|^{\bar{q}}\leq C(\Delta^{\bar{q}}\alpha(\Delta)^{2\bar{q}}\vee\Delta^{\bar{q}(\gamma\wedge\sigma)}).
	\end{equation*}
The desired result (\ref{lare1}) follows. Then combining Lemma \ref{le31} and (\ref{lare1}) gives (\ref{lare2}).
The proof is complete.
\end{proof}

\begin{rem}
	Similar to Remark \ref{remar2}, the $(1-\varepsilon)$-order convergence rate in $\mathcal{L}^{\bar{q}}(\bar{q}>2)$ sense can be given if some stronger conditions are added.
\end{rem}

\section{Numerical example}

In this section, a numerical example is shown to test our theory.
Consider nonautonomous and highly nonlinear  SDDE
\begin{equation}\label{exx1}
		\begin{split}
	d x(t)=&\left(\frac{1}{8}|x(t-\tau)|^{\frac{5}{4}}-5x^{3}(t)+2\zeta_{t}x(t)\right) d t\\
	&+\left(\frac{1}{2}|x(t)|^{\frac{3}{2}} +\zeta_{t}x(t-\tau)\right)d B(t), \quad t \in [0,1],
		\end{split}
\end{equation}
with the initial data $\xi$ which satisfies Assumption \ref{a35} with $\gamma=1$. Here, $\zeta_{t}=[t(1-t)]^{\frac{3}{4}}$ and $B(t)$ is a scalar Brownian motion.
It is easy to
verify  Assumption \ref{a31}, so we omit it. Furthermore, we get that
\begin{equation*}
	\begin{split}
		&( x-\bar{x})^T(f(t,x,y)-f(t,\bar{x},\bar{y}))
		\\\leq&3|x-\bar{x}|^{2}+5|x-\bar{x}|^{2}(-\frac{1}{2}(x^{2}+\bar{x}^{2}))+\frac{25}{256}|y-\bar{y}|^{2}(|y|^{\frac{1}{4}}+|\bar{y}|^{\frac{1}{4}})^{2}
		\\\leq&3|x-\bar{x}|^{2}+\frac{25}{64}|y-\bar{y}|^{2}-\frac{5}{2}|x-\bar{x}|^{2}(|x|^{2}+|\bar{x}|^{2})+\frac{25}{128}|y-\bar{y}|^{2}(|y|^{2}+|\bar{y}|^{2}).
	\end{split}
\end{equation*}
Let $q=2$. Similarly,
\begin{equation*}
	\begin{split}
		(q-1)|g(t, x,y)-g(t,\bar{x}, \bar{y})|^{2}\leq&\frac{1}{2}||x|^{\frac{3}{2}}-|\bar{x}|^{\frac{3}{2}}|^{2}+2|y-\bar{y}|^{2}
		\\\leq&\frac{9}{2}|x-\bar{x}|^{2}+2|y-\bar{y}|^{2}+\frac{9}{4}|x-\bar{x}|^{2}(|x|^{2}+|\bar{x}|^{2}).
	\end{split}
\end{equation*}
Combining the above two inequalities together yields that
\begin{equation*}
	\begin{split}
		( x-&\bar{x})^T(f(t,x,y)-f(t,\bar{x},\bar{y}))+(q-1)|g(t, x,y)-g(t,\bar{x}, \bar{y})|^{2}
		\\\leq&8|x-\bar{x}|^{2}-\frac{1}{4}|x-\bar{x}|^{2}(|x|^{2}+|\bar{x}|^{2})+8|y-\bar{y}|^{2}+\frac{1}{4}|y-\bar{y}|^{2}(|y|^{2}+|\bar{y}|^{2}).
	\end{split}
\end{equation*}
This means that Assumption \ref{a32} holds with $U(m,n)=\frac{1}{4}|m-n|^{2}(|m|^{2}+|n|^{2})$.
Then for $p>2$, Assumption \ref{a33} can be verified simply.
Additionally, it is easy to show that Assumption \ref{a34} and Assumption \ref{a36} are satisfied with $\beta=2$, $\sigma=\frac{3}{4}$. Then, choose $\Lambda(w)=5w^{3}$ for $w\geq1$ and $\alpha(\Delta)=\Delta^{-\frac{1}{8}}$. By Theorem \ref{theo37}, one can see that
\begin{equation*}
	\mathbb{E}|x(T)-\eta_{\Delta}(T)|^{2}\leq C\Delta^{\frac{3}{2}}.
\end{equation*}

\begin{figure}[htbp]
	\centering
	\includegraphics[height=6.8cm,width=8cm]{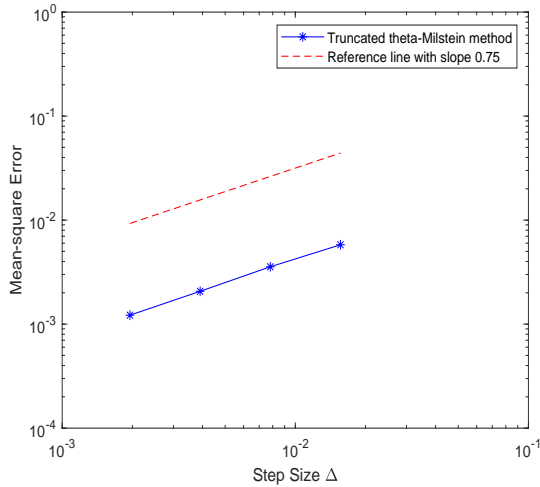}
	\caption{\label{4} The convergence rate of the truncated $\theta$-Milstein method for (\ref{exx1}).}
\end{figure}

Note that we cannot give the explicit form of the true solution. In the numerical experiment, we view the truncated $\theta$-Milstein method with the smallest step size ($2^{-11}$) as the true solution. Two thousand sample paths are simulated. Figure 1 gives the $\mathcal{L}^{2}$-errors defined by
\begin{equation*}
	(\mathbb{E}|x(T)-\eta_{\Delta}(T)|^{2})^{\frac{1}{2}}\approx \left(\frac{1}{2000}\sum_{i=1}^{2000}|x^{i}(T)-\eta^{i}_{\Delta}(T)|^{2}\right)^{\frac{1}{2}},
\end{equation*}
with step sizes $2^{-9}$, $2^{-8}$, $2^{-7}$, $2^{-6}$ at $T=1$.
 From Figure 1, we can observe that the convergence rate of the truncated $\theta$-Milstein method for (\ref{exx1}) is approximately 0.75, which implies that the theoretical results are reliable.

\begin{rem}
	When delay is vanishing (i.e., $\tau =0$), SDDE (\ref{exx1}) will degenerate into nonautonomous SDE. The research about the explicit TMM (i.e., $\theta =0$) for nonautonomous SDEs can be found in \cite{23}.
\end{rem}

\section{Conclusion and future research}

In this work,  we establish the truncated $\theta$-Milstein method for a class of nonautonomous and highly nonlinear SDDEs   which have practical applications in many fields.
The convergence rate is investigated in $\mathcal{L}^{\bar{q}}(\bar{q}\geq2)$ sense under the weaker conditions than the existing results. An example and its numerical simulation are presented to show the effectiveness of the truncated $\theta$-Milstein method.

The future work is to analyze the long-time asymptotic behavior of the truncated $\theta$-Milstein method for SDDE (\ref{sdde1}). 
Additionally, investigating the stability
of the truncated $\theta$-Milstein method for SDDE (\ref{sdde1}) by adjusting the parameter $\theta$ is meaningful to us.

\section*{Acknowledgements}

This work is supported by the National Natural Science Foundation of China (Grant Nos. 61876192 and 11871343) and Science and Technology Innovation Plan of Shanghai, China (Grant No. 20JC1414200).

\end{document}